\documentclass[11pt]{amsart}

\usepackage[latin2]{inputenc}
\usepackage{graphicx}

\usepackage[frozencache=true,cachedir=.]{minted}

\usepackage{tikz,pgf}

\usetikzlibrary{calc,intersections,through,backgrounds,patterns,arrows.meta, math,through}
\usepackage{pgfplots}
\usepackage[labelformat = brace, position=top]{subcaption}

\usepackage{mathtools,amsfonts,esint,amsthm,amssymb,amsmath,enumerate}

\usepackage[english]{babel}

\newtheorem{theorem}{Theorem}[section]
\newtheorem{proposition}[theorem]{Proposition}
\newtheorem{lemma}[theorem]{Lemma}
\newtheorem{exercise}[theorem]{Exercise}
\newtheorem{corollary}[theorem]{Corollary}
\newtheorem{definition}[theorem]{Definition}

\newtheorem{remark}[theorem]{Remark}
\newtheorem*{theorem*}{Theorem}
\theoremstyle{definition}
\newtheorem{example}{Example}[section]

\newenvironment{hint}{\begin{small} \emph{Hint:}}{\end{small}}

\definecolor{Yellow}{rgb}{0.95,0.9,0.0} 
\definecolor{Red}{rgb}{0.8,0.1,0.1}
\definecolor{Green}{rgb}{0.1,0.65,0.2}
\definecolor{Blue}{rgb}{0.1,0.1,0.8}
\definecolor{Purple}{rgb}{0.7,0.1,0.7}
\definecolor{Grey}{rgb}{0.6,0.6,0.6}

\definecolor{YELLOW}{rgb}{0.95,0.9,0.0} 
\definecolor{RED}{rgb}{0.8,0.1,0.1}
\definecolor{GREEN}{rgb}{0.25,0.65,0.1}
\definecolor{BLUE}{rgb}{0.1,0.1,0.8}
\definecolor{PURPLE}{rgb}{0.7,0.1,0.7}

\usepackage[linktocpage=true,colorlinks=true,linkcolor=BLUE,citecolor=GREEN]{hyperref}

\newcommand{\supp}{\operatorname{supp}} 
\renewcommand{\div}{\operatorname{div}} 
\newcommand{\sign}{\operatorname{sign}}

\newcommand{\dist}{\operatorname{dist}} 
\newcommand{\sdist}{\operatorname{sdist}} 
\DeclareMathOperator*{\esssup}{ess\,sup}

\newcommand{\trace}{\operatorname{tr}} 
\newcommand{\Lip}{\operatorname{Lip}} 
\newcommand{\const}{\operatorname{const.}}
\newcommand{\diag}{\operatorname{diag}}

\newcommand{\R}{\mathbb{R}}
\newcommand{\N}{\mathbb{N}}
\newcommand{\sphere}{\mathbb{S}}
\renewcommand{\H}{\mathcal{H}}

\newcommand{\eps}{\varepsilon}
\allowdisplaybreaks[2]

\newcommand{\p}{\partial}
\newcommand{\dL}{\,\mathrm{d}}
\newcommand{\ddt}{\frac{\mathrm{d}}{\mathrm{d}t}}

\begin{document}

\title[Distributional solutions to mean curvature flow]{Distributional solutions to mean curvature flow}

\author{Tim Laux}
\address{Tim Laux, Hausdorff Center for Mathematics, University of Bonn,  Villa Maria, Endenicher Alllee 62, 53115 Bonn, Germany}
\email{tim.laux@hcm.uni-bonn.de}

\begin{abstract}
	These lecture notes aim to present some of the ideas behind the recent (conditional) existence and (weak-strong) uniqueness theory for mean curvature flow. 
	Focusing on the simplest case of the evolution of a single closed hypersurface allows for a self-contained and concise presentation, which is accessible for beginning graduate students with some background in PDEs and only requires basic measure theory. 
\end{abstract}

\maketitle
\tableofcontents

%
%

\section{Motivation, (numerical) examples \& classical theory}

	Mean curvature flow (MCF) is one of the two most elementary geometric evolution equations and arises in many problems from physics and engineering to biology and chemistry, and of course plays an important role in geometry. 
	MCF is the most natural diffusion equation for surfaces in terms of their extrinsic geometry and is therefore closely related to the Ricci flow, the natural diffusion equation in terms of the intrinsic geometry. 
	Just like any diffusion equation, it instantly regularizes the initial conditions. However, after finite time, topology changes may occur.
	The classical solution concept relies on parametrizing the evolving surfaces over a fixed reference manifold, and is therefore only applicable before the onset of these topology changes.
	Weak solutions on the other hand allow to describe the evolution through these singular events. 
	
	There are several concepts of weak solutions to MCF. These notes focus on the distributional solution introduced by Luckhaus--Sturzenhecker \cite{LucStu95}. 
	This first section gives a somewhat broad and very basic introduction to mean curvature flow. Here, some aspects of the classical theory are recalled, mainly focusing on simple computations in order to get familiar with the equation. The more experienced reader will gladly skip this section.
	
	In Section \ref{sec:existence}, the construction of Almgren--Taylor--Wang \cite{ATW93} and Luck\-haus--Sturzenhecker \cite{LucStu95} is discussed; and a conditional closure theorem is proven to illustrate at the simplest example the structure of the proof in \cite{LucStu95} and other recent proofs inspired by that work \cite{LauxLelmi,LauxOtto, LauxOttoBrakke, LauxSimon}. 
	For an accessible introduction to Brakke's mean curvature flow, the interested reader is referred to the recent monograph of Tonegawa \cite{TonegawaBook}, and for the concept of viscosity solutions, one may in fact refer directly to the original paper by Evans and Spruck \cite{EvansSpruckI}. 
	
	Section \ref{sec:uniqueness} is devoted to uniqueness issues for the distributional solution. 
	Here again, we focus on the simple two-phase case, but again, the method carries over to the multiphase case. The interested reader is referred to the original paper \cite{FHLS2D} for this generalization. We also refer to \cite{JerSme} for the case of the binormal curvature flow of filaments in $\R^3$, which can be viewed as the Schr\"odinger analog of our case here. 
	
 	For the purpose of these notes, a \emph{classical mean curvature flow} is a one-parameter family of hypersurfaces $(\Sigma(t))_{t\in[0,T)}$ satisfying 
 	\begin{align}\label{eq:intro VH}
	 	V = -\mu \sigma H \quad \text{on }\Sigma(t).
 	\end{align}
	Here $V$ denotes the normal velocity, $H$ denotes the mean curvature, and $\mu,\sigma>0$ are  fixed parameters, called mobility and surface tension, respectively. 
 	We will always use the traditional sign convention that $V>0$ for expanding enclosed volumes and $H>0$ for convex enclosed volumes.
	
	\subsection{Physical motivation and gradient-flow structure}
	Before getting into gory mathematical details, let's try to motivate MCF as a phenomenological model for evolution of phase boundaries or other interfaces in physical models.
	Let's focus on grain growth, the slow relaxation of grain boundaries in polycrystals. At sufficiently high temperature, the grain boundary network slowly moves in order to relax the overall energy. 
	A crucial first observation is that a grain boundary is an imaginary surface, and not a thin region of evolving particles (as for example a soap bubble). 
	The grain boundary simply separates two regions, so-called ``grains'', of the same material and each of a different constant crystal orientation.
	This idealized boundary moves because atoms close to this imaginary surface, which lie on one of the lattices decide to orient along the other lattice because they feel their neighbors on the other side of the grain boundary. 
	The upshot is that no material is transported on the scale at which the grain boundary moves.
	In other words, we are not interested in the trajectories of points on the grain boundary but of the change of location and shape of the whole grain boundary.
	In the language of fluid dynamics, the former situation would correspond to the Lagrangian viewpoint, the latter to the Eulerian viewpoint.
	In these notes, we will mostly adopt the latter viewpoint. 
	
	The energy which is carried by a grain boundary depends on the mismatch of the two crystal lattices and in general also on the orientation of the grain boundary with respect to these two lattices. 
	In these notes, we will focus on the simplest geometric situation of one grain which is embedded in a much larger grain. 
	Then we only deal with two crystal lattice, and for simplicity, we assume that the energy density does not depend on the orientation of the grain boundary, which can be justified if the mismatch angle between the two crystal lattices is small \cite{read1950dislocation}.
	In this case, the total energy is simply given by the area functional
	\begin{align}
		E(\p \Omega) = \sigma \mathcal{H}^{d-1}(\p \Omega),
	\end{align}
	where $\Omega \subset \R^d $ is the smaller grain and $d\geq 2$ denotes the dimension of the ambient space and $\mathcal{H}^{d-1}$ denotes the $(d-1)$-dimensional Hausdorff measure. 
	The rate of change of this surface area can be computed as\footnote{This and other computations will be justified later on in more detail.} 
	\begin{align*}
		\ddt E(\p \Omega(t)) = \sigma \int_{\p \Omega(t)} VH \dL \mathcal{H}^{d-1},
	\end{align*}
	where $V$ denotes the normal velocity of the grain boundary and $H$ denotes its mean curvature (with the standard sign conventions $V>0$ for expanding $\Omega$ and $H>0$ for convex $\Omega$). 
	If $\p \Omega(t)$ now moves by MCF  \eqref{eq:intro VH}, then 
	\begin{align*}
			\ddt E(\p \Omega(t)) = - \frac1\mu \int_{\p \Omega(t)} V^2 \dL \mathcal{H}^{d-1} \leq 0,
	\end{align*}
	from which we can read off that MCF is the gradient flow of the energy $E$ with respect to the $L^2$-norm on $\p \Omega$, weighted by the inverse mobility $\frac1\mu$.
	
	Just as our grain boundary was not made up of particles moving with the surface, also this equation does not see individual particles: it is geometric in the sense that it is invariant under reparametrizations of the interface $\p \Omega(t)$ since the normal velocity and the mean curvature do not depend on the choice of parametrization.
	Although interesting for applications---and in particular for more difficult systems---, the two parameters $\mu$ and $\sigma$ can be scaled out in our present context, so that we will from now on assume $\mu=\sigma=1$. 
	
	\subsection{Planar curves}
	\label{sec:planar curves}
	
	To get acquainted with MCF, let us consider the simplest case of planar curves, in which case MCF is called the \emph{curve shortening flow}. 
	Throughout this section we will neglect regularity issues and assume that all maps and functions will be sufficiently smooth to carry out all manipulations.
	We seek to find a one-parameter family of embedded parametrized curves $(X(\,\cdot\,,t))_{t\in [0,T)}$, i.e., $X\colon \sphere^1\times[0,T)\to \R^2$ such that $\p_u X(u,t) \neq 0 $ for all $u\in \sphere^1$ and all $t\in [0,T)$, which satisfies the PDE
	\begin{align}\label{eq:PDE deg CSF}
		\nu \cdot \p_t X =- \kappa
	\end{align}
	and attains the initial datum
	\begin{align}
		X(\,\cdot\,,0) = X_0.
	\end{align}
	Here $X_0$ is some given parametrized curve $X_0 \colon \sphere^1 \to \R^2$, $\kappa$ denotes the curvature of the curve $X(\,\cdot\,,t)$ and $\nu$ is the unit normal defined below.
	
	In the following exercise, we'll show that \eqref{eq:PDE deg CSF} is ``geometric'' in the sense that it does not depend on the particular parametrization of the curves $X(\,\cdot\,,t)$.
	\begin{exercise}\label{exercise:reparametrization}
		Let $\varphi\colon \sphere^1\times (0,T) \to \sphere^1$ be a one-parameter family of orien\-ta\-tion-preserving diffeomorphisms, i.e., $\varphi$ is smooth and $\p_v \varphi(v,t) >0$ for all $v\in \sphere^1$ and $t\in (0,T)$.
		Show that if $X$ solves the PDE \eqref{eq:PDE deg CSF}, then so does $Y=X\circ \varphi$.		
	\end{exercise}
	
	Therefore, we can always reparametrize to get a \emph{normal parame\-triza\-tion} such that $\partial_t X \cdot \tau =0$; then \eqref{eq:PDE deg CSF} reads 
	\begin{align}\label{eq:PDE nondeg CSF}
		\p_t X = - \kappa \nu.
	\end{align}

	To compute the curvature recall that we can parametrize any curve by arc-length. Although the arc-length parameter of a closed curve is only unique up to an additive constant, the length element and derivative w.r.t.\ arc length are uniquely determined by the relations
	\begin{align*}
	\dL s = |\p_u X| \dL u\quad \text{and} \quad \p_s = \frac1{|\p_u X|} \p_u.
	\end{align*}
	Then the tangent $\tau$, normal $\nu$, and curvature $\kappa$ are given by
	\begin{align}\label{eq:curves def normal etc}
		\tau = \p_s X, \quad \nu=J\tau,\quad \kappa = -\p_{s}^2 X \cdot \nu =- \p_s \tau \cdot \nu,
	\end{align}
	where $J = \begin{pmatrix} 0 & -1\\ 1& 0 \end{pmatrix}$ is the counter-clockwise rotation by $90^\circ$.\footnote{Note that for a positively oriented simple closed curve, \eqref{eq:PDE deg CSF} and \eqref{eq:curves def normal etc} follow the same sign convention as described in \eqref{eq:intro VH}.} 
	Therefore the curve shortening flow equation \eqref{eq:PDE nondeg CSF} simply reads 
	\begin{align}\label{eq:curve shortening}
		\partial_t X = \p_{s}^2 X.
	\end{align}
	Formally, this looks like the heat equation, but of course the differential operator $\p_{s}^2$ on the right-hand side depends on the parametrization $X$ itself, which makes the equation nonlinear.
	
	\begin{exercise}\label{exercise:commutator}
		If $X$ solves \eqref{eq:curve shortening}, then the length element satisfies
		\begin{align}\label{eq:dt length element}
		\p_t \dL s = - \kappa^2 \dL s
		\end{align}
		and we have the commutator relation
		\begin{align}\label{eq:commutator dt ds}
		[\p_t,\p_s] := \p_t\p_s-\p_s\p_t = \kappa^2 \p_s.
		\end{align}
	\end{exercise}
	
%
%

		%

		\begin{proof}[Solution to Exercise \ref{exercise:commutator}]
			To prove \eqref{eq:dt length element}, we compute
			\begin{align*}
				\p_t |\p_u X| 
				= \frac{\p_u X\cdot \p_t\p_uX}{|\p_uX|} 
				= \frac{\p_u X}{|\p_u X|} \cdot \p_u \p_t X 
				&\stackrel{\eqref{eq:curve shortening}}{=}-  \tau \cdot \p_u (\kappa \nu)
				\\&=- \kappa \tau \cdot \p_u \nu 
				= \kappa^2 |\p_u X|.
			\end{align*}
			
			Now let us turn to \eqref{eq:commutator dt ds}. 
			For $f \in C^1([0,T)\times \sphere^1)$ it holds 
			\begin{align}
				[\p_t,\p_s] f 
				 = \p_t \left( \frac1{|\p_u X|} \p_u f\right) -  \frac1{|\p_u X|} \p_u \p_t f
				 = \p_u f \p_t \frac1{|\p_u X|}.
			\end{align}
			By \eqref{eq:dt length element} and the chain rule,
			\begin{align}
				\p_t \frac1{|\p_u X|} 
				= -\frac1{|\p_u X|^2} \p_t |\p_u X|
				\stackrel{\eqref{eq:dt length element}}{=} -\frac1{|\p_u X|^2} (-\kappa^2 |\p_u X|)
				=  \kappa^2\frac1{|\p_u X|}.
			\end{align}
			Hence
			\begin{align}
				[\p_t,\p_s] f = \kappa^2 \p_s f \quad \text{for all } f\in C^1([0,T)\times \sphere^1),
			\end{align}
			which is \eqref{eq:commutator dt ds}.
		\end{proof}
		
		An important consequence of this computation is that the total length of the curve is decreasing along the flow:
		
		\begin{corollary}[Evolution of length]
			\label{cor:ev length}
			The total length of a curve, given by
			\begin{align*}
				L(X) &:= \int_X 1 \dL s = \int_{\sphere^1} |\p_u X| \dL u,
			\end{align*}
			under the evolution \eqref{eq:curve shortening} satisfies
			\begin{align}
				\label{eq:dt Length}
				\ddt L(X(\, \cdot \, , t)) &= - \int_X \kappa^2 \dL s.
			\end{align}	
		\end{corollary}
		
		\begin{proof}
			By \eqref{eq:dt length element}, we have
			\begin{align*}
			\ddt L(X(\, \cdot \, , t)) 
				= \int_X \p_t \dL s
				= -\int_X \kappa^2\dL s
			\end{align*}
			as desired.
		\end{proof}
		
		Also local quantities can now be monitored easily.
		\begin{lemma}[Evolution of geometric quantities]\label{lemma:curves dt tau dt nu dt kappa}
			The tangent and normal vectors satisfy
			\begin{align}
				\label{eq:dt tau} \p_t \tau & = (\p_s \kappa) \nu,
				\\ \label{eq:dt nu} \p_t \nu &= - (\p_s \kappa) \tau.
			\end{align}
			The curvature satisfies the evolution equation
			\begin{align}
				\label{eq:dt kappa} \p_t \kappa - \p_s^2\kappa = \kappa^3.
			\end{align}
		\end{lemma}
		
		\begin{remark}\label{remark:reaction diffusion kappa}
			The simple structure of the evolution equation for the curvature is remarkable: it's a reaction-diffusion equation! 
			
			\begin{enumerate}
				\item First, note that the sign of the reaction-term $\kappa^3$ is favorable for a maximum-principle argument to show
				\begin{align}\label{MP for kappa}
					\kappa(\,\cdot\,,0) \geq 0 \Rightarrow \kappa(\,\cdot\,,t) \geq 0 \quad \text{for }t>0.
				\end{align}
				In other words, convex curves remain convex under CSF.
							
				\item However, in terms of regularity, the reaction term $\kappa^3$ has the ``bad'' sign in the sense that it provokes blow-up in finite time. 
				Such an ``ODE-driven'' blow-up has the rate 
				\begin{align}
					\kappa(t)\sim \frac1{\sqrt{2(T-t)}} \quad \text{as }t\uparrow T.
				\end{align}
			\end{enumerate}
		\end{remark}
	
		\begin{proof}[Proof of Lemma \ref{lemma:curves dt tau dt nu dt kappa}]
			Note first that \eqref{eq:dt nu} follows from applying the rotation $J$ to \eqref{eq:dt tau}. To prove the latter, we use \eqref{eq:curve shortening} and \eqref{eq:commutator dt ds} to get
			\begin{align*}
				\p_t \tau 
				= \p_t \p_s X 
				= \p_s \p_t X + [\p_t,\p_s] X
				\stackrel{\eqref{eq:curve shortening},\eqref{eq:commutator dt ds} }{=}- \p_s( \kappa \nu) + \kappa^2 \p_sX 
				= (-\p_s\kappa) \nu,
			\end{align*}
			where we have used $\p_s \nu =J\p_s \tau = -J \kappa \nu = \kappa \tau$ in the last step. 
			(Of course the $\tau$-component had to cancel since $|\tau|^2 =1$ implies that $\tau \cdot \p_t \tau =0$!)
			
			Now, let's turn to \eqref{eq:dt kappa}. By the definition of $\kappa$ and since $\p_t \nu $ and $ \p_s\tau $ are orthogonal, we have
			\begin{align}
				-\p_t \kappa 
				= \p_t  \left( \nu \cdot \p_s \tau \right)
				= \nu \cdot \p_t\p_s \tau
				=\nu \cdot \p_s \p_t \tau + \nu \cdot [\p_t,\p_s]\tau.
			\end{align}
			Now plugging in the evolution equation \eqref{eq:dt tau} for $\tau$ and the commutator identity \eqref{eq:commutator dt ds}, we obtain
			\begin{align}
				-\p_t \kappa
				= \nu \cdot \p_s ((\p_s\kappa) \nu)+ \kappa^2 \nu \cdot \p_s \tau
				= \p_s^2 \kappa +\kappa^3,
			\end{align}
			which is precisely \eqref{eq:dt kappa}.
		\end{proof}
		
		\begin{exercise}[Basic identities for arbitrary geometric evolution equations]
			\label{ex:basic identities arbitrary}
				Let $X$ solve a general evolution equation of the form
			\begin{align}\label{eq:curve general evolution}
			\partial_t X  = V \nu
			\end{align}
			for some function $V$. 
			Show that
			\begin{align}
				\p_t \dL s &=  \kappa V \dL s,
			\\	[\p_t,\p_s] & = -\kappa V \p_s,
			\\  \p_t \tau & = -(\p_s V) \nu,
			\\ \p_t \nu &= (\p_s V) \tau,
			\\ 	\p_t \kappa + \p_s^2 V &= -\kappa^2 V.
			\end{align}
		\end{exercise}
	
		\begin{example}
			Other interesting geometric evolution equations are the Willmore flow and surface diffusion. The first appears in the relaxation of elastic energy of thin plates, and the latter when particles diffuse along a surface.
			In our two-dimensional set-up, these would correspond to the choices $V=\p_s^2\kappa +\frac12 \kappa^3$ and $V=\p_s^2 \kappa$, respectively. 
			It is instructive to post-process the results of Exercise \ref{ex:basic identities arbitrary} in this framework. 
			Most interestingly, for surface (or rather curve) diffusion, we obtain 
			\begin{align*}
				\p_t \dL s=\kappa\p_s^2 \kappa \dL s = \left( \p_s(\kappa \p_s \kappa) - (\p_s\kappa)^2 \right) \dL s,
			\end{align*}
			which---as in the proof of Corollary \ref{cor:ev length}---yields 
			\begin{align*}
				\ddt L(X(\,\cdot\,,t)) = -\int_{X(\,\cdot\,,t)} (\p_s\kappa)^2  \dL s\leq 0.
			\end{align*}
			In fact, surface (resp.~curve) diffusion is the $H^{-1}$-gradient flow of the area (resp.~length) functional.
			Finally, observe that for $V=\p_s^2 \kappa$, we have
			\begin{align*}
				\p_t\kappa + \p_s^4 \kappa =-\kappa^2 \p_s^2 \kappa,
			\end{align*}
			which is a parabolic fourth-order equation, but interestingly, the second-order term on the right-hand side has the ``wrong'' sign in terms of regularity.
		\end{example}
	
		The area enclosed by a (positively oriented) simple closed curve is given by
		\begin{align*}
		A(X) &:= \frac12\int_X X \wedge \tau \,\dL s,
		\end{align*}
		where for vectors $X$, $Y$ in $\R^2$, $X\wedge Y  :=X^1Y^2-X^2Y^1= (JX) \cdot Y$.
		The next exercise shows that this enclosed area decreases at a fixed rate, independent of the geometry of the curve.
		 \begin{exercise}[Evolution of enclosed area]\label{exercise:enclosed area}
		 Suppose  $X_0$ is a posivively oriented simple closed curve and let $X(\,\cdot\,,t)$ satisfy \eqref{eq:curve shortening}.
		 Then
			\begin{align} \label{eq:dt Area}
				\ddt A(X(\, \cdot \, , t)) = -2\pi.
			\end{align}
			\begin{hint}
				You can use the one-dimensional version of Gauss-Bonnet $\int_X \kappa \dL s = 2\pi$. 
			\end{hint}
		\end{exercise}
	
		\begin{remark}
			The previous example strongly relies on the fact that we are in dimension $d=2$, since for curves mean curvature and Gau{\ss} curvature are identical. 
			Instead, for MCF in dimension $d=3$, the rate of change of enclosed volume equals $-\int_{\p \Omega}H \dL S$, which has units of length. 
			MacPherson and Srolovitz \cite{MacPhersonSrolovitz} identified this term as $-2\pi$ times the \emph{mean width} of $\Omega$, and showed that this is one of the terms in the more general case of multiphase MCF, which generalizes the famous von Neumann relation
			\begin{align*}
				\ddt A(\Omega(t)) = -2\pi \left(1-\frac n6\right),
			\end{align*}
			which holds for any phase (or ``grain'') $\Omega(t)$ in multiphase mean curvature flow of planar curves. Here $n \in \N$ denotes the number of triple junctions on $\p \Omega$. Our case of Example \ref{exercise:enclosed area} corresponds to the special case $n=0$.
		\end{remark}
		
		Let's finish our discussion of planar curves with a handful of deeper theorems and an illustration. 
		First recall that convexity is preserved thanks to the reaction-diffusion PDE \eqref{eq:dt kappa} for the curvature as we mentioned in Remark \ref{remark:reaction diffusion kappa}. 
		Next, by Exercise \ref{exercise:enclosed area}, we know precisely when the curve will disappear. The following famous result tells us in addition that the curve becomes rounder and rounder the closer we get to the time of disappearance.
		
		\begin{theorem}[Gage--Hamilton \cite{GageHamilton}]\label{theorem:GageHamilton}
			Let $X_0$ be an embedded, closed convex curve.
			Let $(X(\,\cdot\,,t))_{[t\in[0,T)}$ denote the maximal solution to the CSF, where $T=\frac{A(X_0)}{2\pi}$ and $A(X_0)$ denotes the area enclosed by $X_0$. 
			Then the following three statements hold.
			
			\begin{enumerate}[(i)]
				\item  \emph{Convexity.} \label{item:Gage-Hamilton convex}
				The unique CSF $X\colon \sphere^1 \times [0,T) \to \R^2$ starting from $X_0$ is analytic and strictly convex for all $t\in(0,T)$.
				
				\item \emph{Shrinking to a point.} \label{item:Gage-Hamilton point}
				As $t\uparrow T$, $X(\,\cdot\,,t)$ shrinks to a point $p\in \R^2$.
				
				\item \emph{Roundness.} \label{item:Gage-Hamilton roundness}
				The rescaled curves
				\begin{align}
					Y(\, \cdot \, ,t) := \left(\frac{\pi}{A(X(\,\cdot\,,t))}\right)^\frac12 \left( X(\,\cdot\,,t)- p\right)
				\end{align}
				with constant enclosed area $ A(Y(t)) = A(X_0)$ converge smoothly and exponentially fast to the unit circle.
			\end{enumerate}
		\end{theorem}
	
		Finally, we state Grayson's theorem \cite{grayson}, which guarantees that any embedded closed curve will  become convex before it disappears. Then in turn, Theorem \ref{theorem:GageHamilton} applies, so that the combination of these two theorems fully characterizes the evolution of embedded closed curves. The proof of this theorem has been simplified in several papers, most notably by Andrews and Bryan \cite{AndrewsBryan}, who refine Huisken's distance comparison argument \cite{huiskendistance}.
		
		\begin{theorem}[Grayson \cite{grayson}]
			Let $X_0$ be an embedded closed curve and let $(X(\,\cdot\,,t))_{[t\in[0,T)}$ denote the maximal solution to the CSF. Then there exists $T^*\in(0,T)$ such that 
			$X(\,\cdot\,,T^*)$ is convex.
		\end{theorem}
	
		Figure~\ref{fig:spiral} shows a simulation of a classical example (which was probably inspired by Mullins \cite[Fig.~7]{Mullins}) illustrates the beauty of Grayson's theorem. 
		A simplified and somewhat less efficient version of the code I have used to create the pictures can be found in the Appendix \ref{sec:appendix}. I would recommend the reader to try out several examples with this simple code to get familiar with MCF.
		
		\begin{figure}
		\includegraphics[width=2cm]{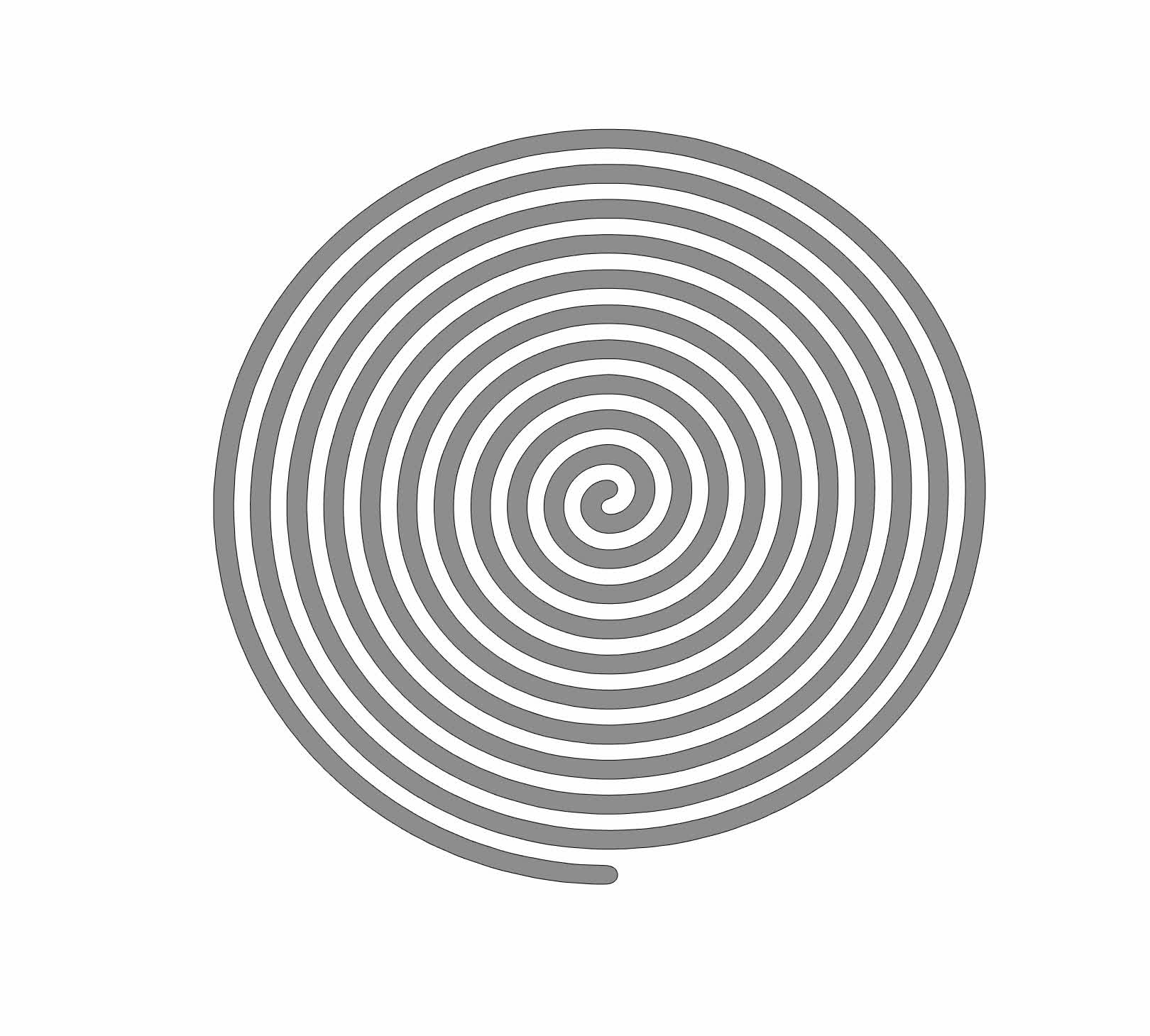}
		\foreach \n in {0200,0400,0600,0800,1000,1200,1400,1600,1800,2000,2200,2400,2600,2800,3000,3200,3400,3600,3800,4000,4200,4400,4600,4800,5000,5200,5400, 5600, 5692} {
				\includegraphics[width=2cm]{thumb\n.jpg}
		}
		\caption{Snapshots of a spiral moving by MCF at equidistant time instances (displayed in lexicographical order). Due to the large curvature of the ``tip'', the spiral unwinds and becomes convex before it collapses. In the last line the curve has completely unwound, then becomes convex, and then rounder, the smaller it gets.}
		\label{fig:spiral}
		\end{figure}

	\subsection{Embedded surfaces and singularities}
	\label{sec:class surfaces}
		
		In this section, we want to get a feeling for what may go wrong in higher dimensions.

	
		Huisken \cite{HuiskenConvex} could generalize Theorem \ref{theorem:GageHamilton} to the case of surfaces: convex surfaces shrink and become rounder and rounder before they disappear. 
		However, Grayson's theorem does not have an analog in higher dimensions: Even smoothly embedded surfaces in $\R^3$ may develop singularities in finite time, and before their extinction as the next example shows.
		
		\begin{example}[Dumbbell: Grayson's example] 
			Think of two large spheres in $\R^3$ connected by a thin tube. 
			The tube will disappear before the spheres have disappeared, creating a ``neck-pinch singularity''. This can be nicely proven by comparison with large balls inside the two spheres and Angenent's torus around the neck, see the next two examples.
		\end{example}
		
		\begin{example}[Spheres]
			Let $X_0(u) = R_0 u$, $u\in \sphere^{d-1}$ for some $R_0>0$. Then MCF becomes a simple ODE for the radius of $X(u,t)=R(t) u$:
			\begin{align*}
				\ddt R = -\frac{d-1}{R}
			\end{align*}
			with the solution $R(t)=\sqrt{R_0^2-2(d-1)t}$, $t\in [0,\frac{R_0^2}{2(d-1)})$.		
		\end{example}
	
		\begin{example}[Angenent's torus]
			There exists a torus $X\colon S^1\times S^1\to \R^3$, which shrinks self-similarly by MCF.
		\end{example}

		We consider a closed $(d-1)$-dimensional manifold $M$ embedded into $\R^d$ by a map $X_0\colon M \to \R^d$. 
		We are interested in the evolution of the parametrized surface $\Sigma_0 := X_0(M)$. 
		In the context of paramterized hypersurfaces, the mean curvature flow equation takes the form
		\begin{align}\label{eq:huisken MCF}
			\p_t X(x,t) =  -H(x,t) \nu(x,t)
		\end{align} 
		denotes the mean curvature vector of the surface $\Sigma(t)= X(M,t)$ at the point $X(x,t)$.
		The unknown here is a one-parameter family of embeddings $X(\,\cdot\,, t) \colon M \to \R^d$. 
		The PDE can be written as
		\begin{align}\label{eq:huisken MCF diffusion}
			\p_t X (x,t) = \Delta_{\Sigma(t)} X(x,t),
		\end{align}
		where $\Delta_{\Sigma(t)}$ denotes the Laplace-Beltrami operator on $\Sigma(t)$.
		That means, we are again solving a diffusion equation.
		As we have already seen in the case of curves, the equation is nonlinear, which here again is due to the fact that the metric on the surface w.r.t.\ which we take the Laplace-Beltrami operator on the right-hand side of \eqref{eq:huisken MCF diffusion} depends on the solution $X$ itself.
		
		Before we continue, let's be more precise about the right-hand side of \eqref{eq:huisken MCF}.
		Here, we follow Huisken's fundamental paper \cite{HuiskenConvex}; let $g=g(\,\cdot\,,t)=(g_{ij})$ denote the metric on $\Sigma(t)$, i.e.,
		\begin{align}
			g_{ij} = \p_i X \cdot \p_j X,
		\end{align}
		in which $\cdot$ denotes the scalar product in $\R^d$ and we supress the dependence on the point $x\in M$ and time $t>0$.
		As usual, with a slight abuse of notation, we denote its inverse matrix by $g^{-1} = (g^{ij})$. 
		Furthermore, we denote by $\nu$ the unit normal vector pointing into the exterior (the unbounded connected component of $\R^d \setminus \Sigma(t)$ (recall $\Sigma(t)$ is assumed to be a closed embedded hypersurface); 
		finally, we denote by $A= (h_{ij})$ the second fundamental form of $\Sigma(t)$, in coordinates
		\begin{align}
			h_{ij} = - \nu\cdot \frac{\p^2 X}{\p x_i \p x_j}.
		\end{align}
		
		Then the scalar mean curvature, with our present sign convention, is simply defined as the trace of $A$ w.r.t.\ the metric $g$, i.e.,
		\begin{align}
			H = \sum_{i,j} g^{ij} h_{ij}
		\end{align}
		and the norm of the second fundamental form will often be denoted by
		\begin{align}
			|A|^2 = \sum_{i,j}h^i_j h^j_i = \sum_{i,j,k,\ell} g^{ik}h_{kj} g^{j\ell}h_{\ell i}.
		\end{align}
		We will use $g^{ij}$ to raise and $g_{ij}$ to lower indices, e.g., $h^i_j = \sum_k g^{ik}h_{kj}$.
		
		Furthermore, we recall the Christoffel symbols
		\begin{align}
			\Gamma_{ij}^k = \frac12 \sum_\ell g^{k\ell} \left( \p_i g_{j\ell} + \p_j g_{i\ell} - \p_\ell g_{ij} \right)
		\end{align}
		and the Levi-Civita connection $\nabla=\nabla_{\Sigma(t)}$ is given by
		\begin{align}\label{eq:def levi civita}
			(\nabla_i Y)^k = \p_i Y^k + \sum_j \Gamma_{ij}^k Y^j
		\end{align}
		for any smooth vector field $Y$.
		The Laplacian $\Delta =\Delta_{\Sigma(t)}$ then is 
		\begin{align}
			\Delta Y = \sum_{i,j} g^{ij} \nabla_i \nabla_j Y.
		\end{align}
		
		\begin{lemma}[Evolution of geometric quantities]
			Let $X$ solve MCF in the parametric setting \eqref{eq:huisken MCF}.
			Then
			\begin{align}
				\p_t g_{ij} &= -2 H h_{ij} \label{eq:huisken dt g}	\\
				\p_t g^{ij} &= 2H h^{ij} \label{eq:huisken dt g inverse}\\
				\p_t \nu &=  -\nabla H	\label{eq:huisken dt nu}	\\
				\p_t h_{ij} &= -\nabla_i \nabla_j H +H \sum_{k,\ell} h_{jk}g^{k\ell}h_{\ell i} \label{eq:huisken dt second ff}
			\end{align}
		\end{lemma}
	
		\begin{proof}
			\emph{Argument for \eqref{eq:huisken dt g}.}
				Since $\nu \cdot \p_j X=0$, we have
				\begin{align}
					h_{ij} = - \nu \cdot \p_i \p_j X = - \p_i (\nu \cdot \p_j X) + \p_i \nu \cdot \p_j X = \p_i \nu \cdot \p_j X.
				\end{align}
				Hence, using once more $\nu \cdot \p_j X=0$, we may compute
				\begin{align*}
					\p_t g_{ij} 
					&= \p_t \left( \p_i X \cdot \p_j X\right) 
					\\&=\p_i (-H\nu) \cdot \p_j X + \p_i X \cdot \p_j (-H\nu).
			 		\\&= -H \p_i \nu \cdot \p_j X -H  \p_i X \cdot \p_j \nu
					\\&= - 2 H h_{ij}.
				\end{align*}
			
			\emph{Argument for \eqref{eq:huisken dt g inverse}.}
				For the remainder of the proof 
				
				Since $\sum_\ell g^{i\ell} g_{\ell k} = \delta^i_k = \const$, we have
				\begin{align*}
					0 = \sum_\ell \p_t (g^{i\ell }g_{\ell k}) 
					= \sum_\ell \p_t g^{i\ell }g_{\ell k} + \sum_\ell g^{i\ell }\p_t g_{\ell k}.
				\end{align*}
				Rearranging and contracting with $g^{kj}$, we obtain
				\begin{align*}
					 -\sum_{k,\ell} g^{i\ell}\p_t g_{\ell k}g^{kj} 
					 =\sum_{k,\ell} \p_t g^{i\ell} g_{\ell k}g^{kj} 
					 = \sum_{k,\ell}  \p_t g^{i\ell}  \delta^j_\ell
					 = \p_t g^{ij}.
				\end{align*}
				Finally, using \eqref{eq:huisken dt g}, we obtain \eqref{eq:huisken dt g inverse}.
				
			\emph{Argument for \eqref{eq:huisken dt nu}.}
				Differentiating $|\nu|^2=1$, we obtain $ \nu \cdot \p_t \nu=0$ and $\nu \cdot \p_i \nu =0$. In particular, $\p_t \nu$ is a tangent vector field and can be expressed in the tangent frame $(\p_i X)_i$. 
				Since $\nu \cdot \p_i X =0 $ we have $\p_t \nu \cdot \p_i X = \p_t (\nu \cdot \p_i X) - \nu \cdot \p_i \p_t X=- \nu \cdot \p_i \p_t X$. 
				Combining these observations, we can simply compute
				\begin{align*}
					\p_t \nu
					&=\sum_{i,j}( \p_t \nu \cdot \p_i X) g^{ij} \p_j X
					\\&= -\sum_{i,j}( \nu \cdot \p_i \p_t X) g^{ij} \p_j X
					\\&= -\sum_{i,j}( \nu \cdot \p_i (-H\nu)) g^{ij} \p_j X
					\\&=\sum_{i,j}\p_i H g^{ij}\p_j X,
				\end{align*}
				which is precisely the tangent vector field $\nabla H$ written in coordinates.

			\emph{Argument for \eqref{eq:huisken dt second ff}.}
			Using $\p_i \nu \cdot \nu =0$, we compute
			\begin{align*}
				\p_t h_{ij} 
				&= - \p_t \left( \p_i\p_j X\cdot \nu \right)
				\\&= \p_i\p_j (H\nu) \cdot \nu - \p_i\p_j X \cdot \p_t \nu
				\\&= \p_i\p_j H +H \p_i\p_j \nu \cdot \nu - \sum_{k,\ell} \p_k H g^{k\ell} \p_i \p_j X \cdot \p_{\ell} X.
			\end{align*}
			Using \eqref{eq:huisken dt nu}, the Gau\ss-Weingarten relations
			\begin{align}\label{eq:huisken Gauss Weingarten}
				\p_j \nu = \sum_{k,\ell} h_{jk} g^{k\ell} \p_\ell X \quad \text{and} \quad 
				\p_i\p_j X = \sum_{m}\Gamma_{ij}^m \p_m X -h_{ij} \nu,
			\end{align}
			and $\p_\ell X \cdot \nu =0$, we obtain	
			\begin{align*}
				\p_t h_{ij} 
				&=  \p_i\p_j H +H \sum_{k,\ell}\p_i\left( h_{jk} g^{k\ell} \p_\ell X \right)\cdot \nu
				- \sum_{k,\ell,m} \p_k H g^{k\ell} g_{m\ell}	\Gamma_{ij}^m				
				\\&=  \p_i\p_j H + H \sum_{k,\ell}  h_{jk} g^{k\ell} \p_i\p_\ell X \cdot \nu
				- \sum_{k,m} \p_k H \delta^{k}_m \Gamma_{ij}^m
				\\&= \p_i\p_j H - \sum_{k}\Gamma_{ij}^k\p_k H  +H \sum_{k,\ell}  h_{jk} g^{k\ell} h_{i\ell}.			
			\end{align*}
			Since $\nabla_j H = \p_j H$ and by the definition of the covariant derivative of a vector field \eqref{eq:def levi civita}, the first two terms combine to the desired $\nabla_i \nabla_j H$.
		\end{proof}
	
		\begin{exercise}
			Derive \eqref{eq:huisken Gauss Weingarten}.
		\end{exercise}
	
		Now we are ready to derive a simple PDE for the mean curvature $H$. In particular, $H>0$ is preserved, which motivates studying MCF of such ``mean convex'' surfaces.
		\begin{corollary}
			If $X$ solves MCF in the parametric setting \eqref{eq:huisken MCF}, then the mean curvature satisfies the evolution equation
			\begin{align}
				\p_t H = \Delta H +|A|^2 H. \label{eq:huisken dt mean curvature}
			\end{align}
			Therefore, if $H>0$ at $t=0$, then $H>0$ at all times $t\in [0,T]$. 
		\end{corollary}
		
		\begin{proof}
				Using \eqref{eq:huisken dt g inverse} and \eqref{eq:huisken dt second ff}, we compute
				\begin{align*}
					\p_t H 
					&= \sum_{i,j} \p_t \left(g^{ij}h_{ij}\right)
					\\&= \sum_{i,j} \p_t g^{ij} h_{ij} + g^{ij} \p_t h_{ij}
					\\&= 2 H \sum_{i,j} h^{ij}h_{ij} + \sum_{i,j} g^{ij} \nabla_i \nabla_j H - H \sum_{i,j,k,\ell} g^{ij} h_{jk} g^{k\ell} h_{\ell i}
					\\&= 2 H|A|^2 + \Delta H - H|A|^2.
				\end{align*} 
				This proves \eqref{eq:huisken dt mean curvature}. 
				The positivity of $H$ is then preserved by the parabolic maximum principle.
		\end{proof}
		
		\begin{exercise}
			Let $X$ solve \eqref{eq:huisken MCF}. Show that the second fundamental form satisfies
			\begin{align}
				\p_t |A|^2 = \Delta |A|^2 -2|\nabla A|^2 +2|A|^4.\label{eq:huisken dt norm of second ff}
			\end{align}
		\end{exercise}
		
		\begin{exercise}
			Let $X \colon M\times [0,T] \to \R^d$ be smooth such that $\p_t X = V \nu$ for some smooth function $V$.  Show that the area element $\dL S := \sqrt{\det g} \dL x$ satisfies
			\begin{align*}
				\p_t\dL S =  VH \dL S 
			\end{align*}
			so that in particular the total surface energy
			\begin{align*}
				E(X) := \int_X 1 \dL S = \int_M \sqrt{\det g(x,t)} \dL x
			\end{align*}
			satisfies 
			\begin{align*}
				\frac{\dL }{\dL t} E(X(\,\cdot\,,t)) = \int_X VH\dL S.
			\end{align*}
			For MCF, this means that the total surface energy satisfies the energy-dissipaton relation
			\begin{align*}
				\frac{\dL }{\dL t} E(X(\,\cdot\,,t)) = -\int_X H^2\dL S.
			\end{align*}
			\begin{hint}
				Parametrize over the surface at a fixed time $t_0$, say, $t_0=0$, and then use the expansion of the determinant $\det(I + t M) = 1 + t \trace(M) +O(t^2)$.
			\end{hint}
		\end{exercise}
	
		The following localized version of optimal energy dissipation was introduced by Brakke \cite{Brakke} to define his weak solutions, which are merely varifolds, see \cite{KimTonegawa} for a refined version of his result.
		
		\begin{exercise}[Brakke's (in-)equality]\label{exercise:brakke}
			Let $X \colon M\times [0,T] \to \R^d$ be smooth and denote by $V$ its normal velocity (in direction of the unit normal vector $\nu$). 
			Show that for any test function $\varphi \in C^1(\R^d\times[0,T])$
			\begin{align*}
				\ddt \int_X \varphi \dL S 
				= \int_X \big( \varphi VH + V\nu \cdot \nabla \varphi + \p_t \varphi \big) \dL S.
			\end{align*}
			In particular, if $X$ satisfies  \eqref{eq:huisken MCF}, then
			\begin{align*}
					\ddt \int_X \varphi \dL S 
				= \int_X \big( -\varphi H^2 - H\nu \cdot \nabla \varphi + \p_t \varphi \big) \dL S.
			\end{align*}
		\end{exercise}
	
	Until now, we have used basic differential geometry to study \eqref{eq:huisken MCF}. However, it is also instructive to keep in mind the following special case of graphs over a hyperplane.
	
	\begin{exercise}[Graphs]
		Show that if $\Sigma(t) $ is the graph of a function $f(\,\cdot\,,t)\colon \omega \subset \R^{d-1} \to \R$, then MCF is equivalent to
		\begin{align*}
		\frac{\p_t f}{ \sqrt{1+|\nabla f|^2}} = \nabla \cdot \left( \frac{\nabla f}{\sqrt{1+|\nabla f|^2}} \right),
		\end{align*}
		or equivalently
		\begin{align*}
		\p_t f = \Delta f - \frac{ \nabla f}{\sqrt{1+|\nabla f|^2}}\cdot  \nabla^2 f \,\frac{ \nabla f}{\sqrt{1+|\nabla f|^2}}.
		\end{align*}
	\end{exercise}

	\begin{remark}
		The previous exercise shows that MCF is not a PDE in divergence form, which makes it seemingly difficult to put derivatives on test functions by integration by parts. 
		A powerful tool in studying (two-phase!) MCF is therefore the viscosity approach, since the theory of viscosity solutions applies for a large class of degenerate elliptic and parabolic PDEs in non-divergence form.
	\end{remark}
	%
	%
	\subsection{The (signed) distance function}
	\label{sec:class signed dist}
	
	Of course, the mean curvature depends on the embedding, so it depends on the extrinsic geometry of $\Sigma(t)$. 
	Let us take this thought a little further and instead of studying the evolving surfaces $\Sigma(t) = \p \Omega(t)$ directly, let us study properties of their signed distance function
	\begin{align}\label{dist eq def sdist}
		s(x,t) := \dist(x,\Omega(t)) - \dist(x,\R^d\setminus \Omega(t)),
	\end{align}
	cf.\ Figure \ref{fig:sdist}. 
	
	\begin{figure}
		\centering
			\begin{tikzpicture}[scale=0.7]
			\draw(-4,0)--(4,0);
			\draw[thick,blue,->] (0,0) -- (1,0) node[below]{$\nu(x_0,t)$};
			\draw  (-2,0) node[above]{$\Omega(t)$} ;
			\draw[circle,fill,inner sep=1pt] (0,0) node[above]{$x_0$};
			\draw  (2,0) node[above]{$\R^d\setminus \Omega(t)$};
			\draw[red] (-3,-3)--(3,3) node[left]{$s(x,t)$};
			\end{tikzpicture}
			$\quad$
			\begin{tikzpicture}[scale=0.7]
				\begin{axis}[ view={0}{90}, xticklabels={,,}, yticklabels={,,},colormap/blackwhite]
					\addplot3[  surf, shader=interp, domain=-2:2, y domain=-2:2] {(x^2+y^2)^0.5};
				\end{axis}
				\draw[thick,blue,->] (5,2.85) node[above]{{\color{black}$x_0$}} -- ++(1,0) node[below]{$\nu(x_0,t)$};
			\end{tikzpicture}

		\caption{The signed distance function $s(x,t)$. Left: The graph of $s(x,t)$ plotted along the normal direction around a point $x_0\in \Sigma(t) = \p \Omega(t)$. Right: Heat plot of $s(x,t)$ for $\Omega(t)$ a disk in $\R^2$.}
		\label{fig:sdist}
	\end{figure}
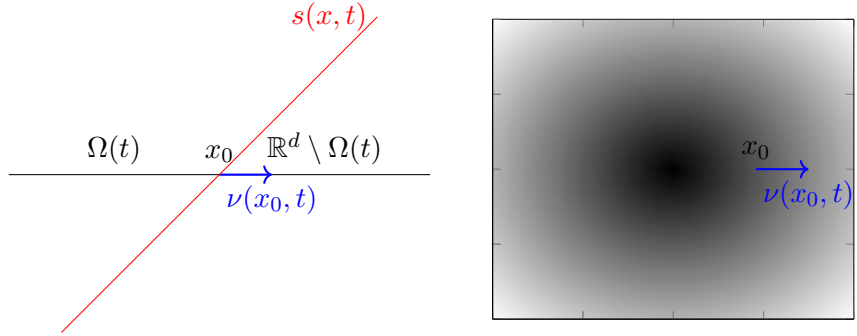
	
	\begin{lemma}\label{lemma:dist transport equation}
		If $\Sigma(t)= \p\Omega(t)$ evolves smoothly with normal velocity $V$ on the time interval $[0,T]$, then there exists $\delta>0$ such that the signed distance function $s$ given by  \eqref{dist eq def sdist} is smooth in the space-time neighborhood
		\begin{align*}
			\mathcal U_\delta := 
			\bigcup_{t\in [0,T]}	B_\delta (\Sigma(t)) \times \{t\} =
			 \left\{ (x,t) \in \R^d \times [0,T] \colon |s(x,t)| \leq \delta\right\}.
		\end{align*}
		Furthermore, it solves the transport equation
		\begin{align*}
			\p_t s + (B \cdot \nabla) s =0 \quad \text{in }\mathcal{U}_\delta,
		\end{align*}
		where the vector field $B$ extends the normal velocity vector field via
		\begin{align*}
			B(x,t) := V(P(x,t),t)
			\,\nu(P(x),t).
		\end{align*}
		Here, $P(\,\cdot\,,t):= P_{\Sigma(t)}$ denotes the nearest point projection from $ 	B_\delta (\Sigma(t)) $ to $\Sigma(t)$ and $\nu(x,t):= \nu_{\Sigma(t)}(x)$ denotes the outward unit normal vector of $\Omega(t)$ at $x\in \Sigma(t)=\p \Omega(t)$.
	\end{lemma}

	Before proving this lemma, let us first consider our case $V=-H$.
	As a direct consequence of the above lemma, we see that $s$ solves the linear diffusion equation on $\Sigma_t$. The next lemma takes this a step further and gives us approximate versions of such equations also away from the interface.
	
	\begin{lemma}\label{lemma:dist transport equation MCF}
		Under the same conditions as in Lemma \ref{lemma:dist transport equation}, if additionally, $\Sigma_t$ is a mean curvature flow, then there exists a constant $C<\infty$ depending on $(\Sigma_t)_{t\in[0,T]}$, such that
		\begin{align}\label{eq:dist transport MCF}
			\left|\p_t s - \Delta s \right| \leq C|s| \quad \text{in }\mathcal{U}_\delta
		\end{align}
		and similarly, the function $\phi := \frac12 s^2$ satisfies 
		\begin{align}\label{eq:dist transport MCF squared}
			\left|\p_t \phi - \Delta \phi  + 1 \right| \leq   C\phi \quad \text{in }\mathcal{U}_\delta.
		\end{align}
	\end{lemma}
	
	\begin{proof}[Proof of Lemma \ref{lemma:dist transport equation}]
		First note that, since we assume $\Sigma(t)$ to be smoothly embedded for all $t\in[0,T]$, since $[0,T]$ is compact, there exists $\delta>0$ such that for each $(x,t)\in \mathcal{U}_\delta$ there exists a unique point $y(t)=P(x,t) \in \Sigma(t)$ such that 
		\begin{align*}
			|s(x,t)| = \min_{y\in \Sigma(t)} |x-y| = |x-y(t)|.
		\end{align*}
		Let $t\in [0,T]$ and let $x\in \R^d \setminus \operatorname{clos} \Omega(t)$ be a point in the exterior, say. Then $s(x,t)>0$, hence $ s(x,t)= |x-y(t)|$.
		Furthermore,
		\begin{align}\label{eq:Ds x-y and nu}
			\nabla s(x,t) = \frac{x-y(t)}{|x-y(t)|}= \nu(y(t),t).
		\end{align}
		Although we don't know how $y(t)$ will evolve along the surface in tangential direction, its velocity in normal direction is completely determined by
		\begin{align*}
			\dot y(t) \cdot \nu(y(t),t) = V(y(t),t).
		\end{align*} 
		Hence, we may simply compute
		\begin{align*}
			\p_t s(x,t) 
			&= -\frac{x-y(t)}{|x-y(t)|} \cdot \dot y(t) 
			= -\nu(y(t)) \cdot \dot y (t)
			= -V(y(t),t) 
			\\&= - B(x,t) \cdot \nu(y(t),t) 
			= -B(x,t)\cdot \nabla s(x,t).\qedhere
		\end{align*}
	\end{proof}

	\begin{proof}[Proof of Lemma \ref{lemma:dist transport equation MCF}]
		\emph{Step 1: Argument for \eqref{eq:dist transport MCF}.}	
		In view of Lemma \ref{lemma:dist transport equation}, in our context of MCF, since $H = \div \nu  = \Delta s $ and $\nabla s =\nu$ on $\Sigma(t)$, we have
		\begin{align}\label{eq:claim H and Delta s}
			\p_t s(x,t) - (\Delta s)(P(x,t),t) =0\quad \text{in } \mathcal{U}_\delta.
		\end{align}
		Recalling the defining property $|x-P(x,t)| = |s(x,t)|$, this implies \eqref{eq:claim H and Delta s} with $C:= \Lip \Delta s$,  the Lipschitz constant of $\Delta s$ on $\mathcal{U}_\delta$.
		
		\emph{Step 2: Argument for \eqref{eq:dist transport MCF squared}.}
		First, note that by chain rule, for any $g\in C^1(\R)$, it holds
		\begin{align*}
			\p_t (g\circ s) +(B \cdot \nabla) (g\circ s) =0 \quad \text{in }\mathcal{U}_\delta.
		\end{align*} 
		Applying this to $g =\frac12 |\cdot |^2$ in the context of the statement of the corollary and then using $\nabla \phi = s\nabla s$, we obtain 
		\begin{align*}
				\p_t \phi(x,t) -s(x,t)(\Delta s)(P(x,t),t) =0 \quad \text{in }\mathcal{U}_\delta.
		\end{align*}
		Now compute $\Delta \phi = \nabla \cdot (s\nabla s) =s\Delta s + |\nabla s|^2 =  s\Delta s +1$, so that
		\begin{align*}
				\p_t \phi   - \Delta \phi 
				= \p_t \phi - s\Delta s -1 
				= -1 + O(s^2)\quad \text{in }\mathcal{U}_\delta,
		\end{align*}
		where the constant in the $O(s^2)$-term is equal to $\Lip \Delta s$.
	\end{proof}
	
	\begin{remark}
		One can show even more: the Hessian matrix $\nabla^2 s(x,t)$ can be written as
		\begin{align*}
			\nabla^2 s = \diag\left( \frac{-\kappa_1}{1-s\kappa_1}, \ldots, \frac{-\kappa_{d-1}}{1-s\kappa_{d-1}},0 \right) \quad \text{in } \mathcal{U}_\delta.
		\end{align*}
	\end{remark}

	\begin{exercise}
		Using $\p_t \phi = 0 $, $\Delta \phi =1$ on $\Sigma(t)$, reprove \eqref{eq:dist transport MCF squared} by a simple Taylor expansion of the function $\p_t \phi -\Delta \phi$.
		
		\begin{hint}
			Compute explicitly the first-order term $\p_t\nabla \phi - \Delta \nabla \phi$ in this Taylor expansion.
		\end{hint}
	\end{exercise}

	\begin{exercise}
		Define $\xi:= \nabla s$ in $\mathcal{U}_\delta$. Which PDE does $\xi$ solve?
	\end{exercise}
	We will come back to this exercise in the last section.
	
	\subsection{Comparison principle and level set formulation}
	\label{sec:visc motivation}
	We start with the crucial observation that two-phase MCF exhibits a geometric comparison principle:
	\begin{theorem}
		``Nested sets remain nested under MCF.''
	\end{theorem}
	The idea of proof is quite instructive:	Let's suppose $\Sigma(t) = \p \Omega(t)$ and $\tilde \Sigma(t) = \p \tilde \Omega(t)$ are two closed manifolds smoothly evolving by MCF such that $\Sigma(0)$ encloses $\tilde \Sigma(0)$ in the sense that $\tilde \Omega(0) \subset \subset \Omega(0)$.
	We want to show that this remains true for all future times $t>0$ (at least as long as both evolutions remain smooth). 
	
	For a contradiction, let us assume that $t_0$ is the first time that $\Sigma(t)$ and $\tilde \Sigma(t)$ touch, say at a point $x_0 \in  \Sigma(t_0) \cap \tilde \Sigma(t_0)$ (by compactness, such a point $x_0$ exists). 
	Then at $x_0$, just before time $t_0$, $\Sigma(t)$ must have moved inwards faster than $\tilde \Sigma(t)$. Since both surfaces were supposed to move by MCF, this implies a relation for their mean curvatures: $H = -V > -\tilde V = \tilde H$ at $(x_0,t_0)$, a contradiction to the fact that $\Sigma(t)$ touches $\tilde \Sigma(t)$
	from the outside, which implies $H \leq \tilde H$.
	
	\medskip
	
	Heuristically, this means, given a function $g \colon \R^d \to \R$, we could in principle let each level set $\{ g=s\} = \{x\in \R^d \colon g(x)=s\}$ ($s \in \R$) evolve by MCF and then reconstruct a function $u(x,t)$ from these level sets by the defining property that $\{u(\,\cdot\,,t) =s\} $ is the evolution by MCF of $\{g=s\}$ at time $t$. 
	This formal correspondence between the level sets and the function $u$ motivates us to \emph{define} weak solutions to MCF with initial datum $\Sigma_0$ by first defining a function $g$ such that $\Sigma_0 = \{g=0\}$ and then setting $\Sigma(t) := \{u(\,\cdot\,,t)=0\}$ to be the corresponding level set of $u$. 
	Formally, it is straightforward to see that $u$ solves
	\begin{align}\label{eq:level set pde}
	\p_t u = |\nabla u|  \div \left( \frac{\nabla u}{|\nabla u|}\right) = \Delta u - \frac{\nabla u}{|\nabla u|}\cdot\nabla^2 u \, \frac{\nabla u}{|\nabla u|}.
	\end{align}
	We will make this a bit more precise in the following paragraph and leave the proof (in the nice regular case) as an exercise, see Exercise \ref{exercise level set pde}.
	Note carefully that \eqref{eq:level set pde} is a quasilinear (degenerately parabolic) PDE in \emph{non-divergence} form.
	
	Let us suppose that $u$ is smooth, and more crucially, that $s=0$ is a regular value of $u$ in the sense that $\nabla u(x,t) \neq 0$ on $\Sigma(t)=\{u=0\}$.
	Under these conditions, \eqref{eq:level set pde} holds if and only if
	\begin{align}\label{eq: levelsets by MCF}
	\text{all level sets of $u$ move by MCF.}
	\end{align}
	
	\begin{exercise}\label{exercise level set pde}
		Show (under the above assumptions on $u$) that \eqref{eq:level set pde} is indeed equivalent to \eqref{eq: levelsets by MCF}.
		
		\begin{hint}
			By the implicit function theorem, $\Sigma(t) $ is a smooth hypersurface with normal
			\[
			\nu(x,t):= \frac{\nabla u(x,t)}{|\nabla u(x,t)|}.
			\]
			Since the mean curvature of the level set $\Gamma(t)$ satisfies $H= \div \nu$, any trajectory $x(t)$ on $ \bigcup_{t\geq 0} (\Sigma(t) \times \{t\})$ has to solve the non-autonomous ODE
			\begin{align}\label{eq:ode x on level set}
			\nu(x,t) \cdot \dot x =  -\div \nu (x,t).
			\end{align}
		\end{hint}
	\end{exercise}
	
	\begin{remark}
		Since \eqref{eq: levelsets by MCF} does not depend on the labeling of the level sets, it is clear that the PDE should satisfy the following invariance: If $\Psi \colon \R \to \R$ is smooth and $u$ is a solution of \eqref{eq:level set pde}, then so is $\Psi \circ u$.
	\end{remark}
	
	\begin{exercise}\label{exercise:psi of u}
		Check by a direct computation that $\Psi\circ u$ is indeed a solution of \eqref{eq:level set pde} if $\Psi\colon \R\to \R$ is smooth with $\Psi'>0$ and $u$ is a smooth solution with $\nabla u\neq0$.
	\end{exercise}

	\begin{remark}
		The theory of viscosity solutions allows to make these ideas rigorous, also in the degenerate case when $\nabla u=0$ somewhere. 
		We refer to Evans--Spruck \cite{EvansSpruckI} and Chen--Giga--Goto \cite{ChenGigaGoto} for more details. 
		In fact, \eqref{eq:level set pde} has a unique viscosity solution. Furthermore, Evans and fSpruck \cite{EvansSpruckIV} showed that almost every level set of $u$ is in fact a unit-density Brakke-flow.
		In fact, one can even show that almost every level set is a distributional solution in the sense of the next section.
	\end{remark}
	
	\subsection{Non-uniqueness and fattening}
	The viscosity solution is unique, however, it may not give us all the information we would like to have as the following example shows.
	
	\begin{example}[Symmetric cross]
		Let 
		\begin{align}
		\Omega_0 := \{x \in \R^2 \colon x_1x_2 >0\}
		\end{align}
		be the union of the first and third quadrants, so that 
		\begin{align}
		\Sigma_0 = \p \Omega_0 = \{x\in \R^2 \colon x_1x_2=0\}
		\end{align}
		is simply given by the union of the two coordinate axes.
		Then the viscosity solution ``fattens'' in the sense that
		\begin{align}
		\{x\in \R^d \colon u(x,t)=0\} \text{ has non-empty interior for all } t>0.
		\end{align}
	\end{example}
	
	\subsection{Numerical schemes}\label{sec:numerical schemes}
	There are plenty of ways to compute solutions to MCF. The simplest and most direct way is to directly discretize the evolution equation for the embeddings $X(u,t)$.
	However, this breaks down when topological changes occur and is much more complicated in the multiphase setting (to simulate more than two grains). 
	To handle topological changes in two-phase systems, one can use the levelset method \cite{OsherSethian}. 
	Other popular schemes to handle topological changes and multiphase systems are Monte Carlo methods which are probabilistic in nature; phase-field methods; and the thresholding scheme \cite{MBO}.
	
	The idea of phase-field methods is in some sense also of physical nature: in many phase-field models like the Allen--Cahn equation
	\begin{align}\label{eq:AC}
		\partial_t u_\eps = \Delta u_\eps + \frac1{\eps^2}u_\eps (1-u_\eps^2),
	\end{align}
	it is observed after a very fast transition that a diffuse interface of approximate thickness $\eps\ll 1$ forms, which moves according to MCF, see Figures \ref{fig:allencahnfast} and \ref{fig:allencahnslow}.
	This asymptotic behavior can be explained easily by matched formal asymptotic expansion, and there are many rigorous convergence proofs in the literature. 
	In the computational phase-field community, this phenomenon is often described in a reciprocal fashion: computing solutions to \eqref{eq:AC} gives a good (numerical) approximation to mean curvature flow. 
	This is particularly interesting in the presence of topological changes, which are not seen by the reaction-diffusion equation \eqref{eq:AC} and therefore appear naturally.
	Numerically, the main drawback is that in order to accurately solve \eqref{eq:AC}, one needs to either use a fine grid or an adaptive grid in order to resolve the fast transition from $-1$ to $1$ in the diffuse transition layer.
	
	\begin{figure}
		\foreach \n in {1, 3, 5, 7, 9} {
			\includegraphics[width=2cm]{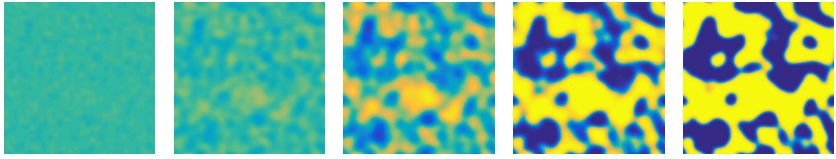}
		}
		\caption{Snapshots of a the solution $u_\eps$ of the Allen--Cahn equation with random initial data and periodic boundary conditions during the first stage forming diffuse interfaces.}
		\label{fig:allencahnfast}
	\end{figure}

	\begin{figure}
		\includegraphics[width=2cm]{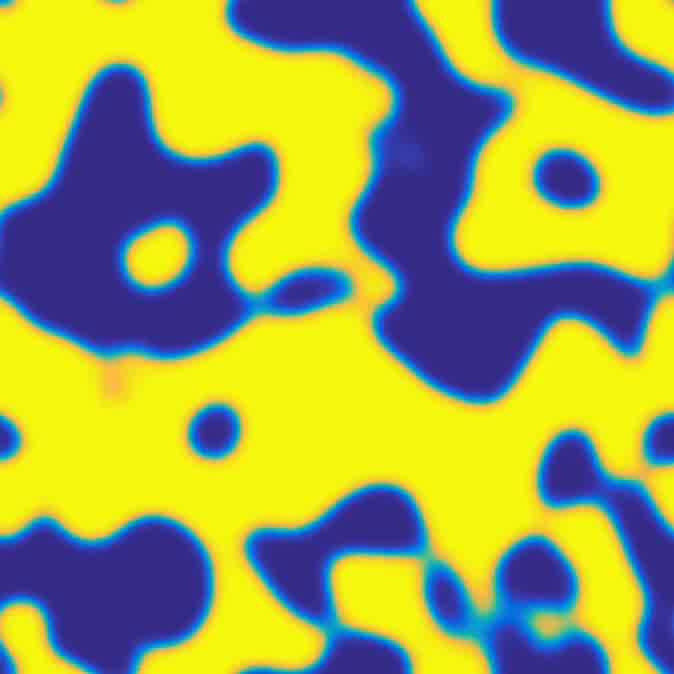}
		\foreach \n in {5,9} {
			\includegraphics[width=2cm]{random00\n0.jpg}
		}%
		\foreach \n in {3,7} {
			\includegraphics[width=2cm]{random01\n0.jpg}
		}%
		\foreach \n in {1,3,7} {
			\includegraphics[width=2cm]{random02\n0.jpg}
		}%
		\foreach \n in {1,3,7,9} {
			\includegraphics[width=2cm]{random03\n0.jpg}
		}%
		\foreach \n in {0,4} {
			\includegraphics[width=2cm]{random04\n0.jpg}
		} %
		\includegraphics[width=2cm]{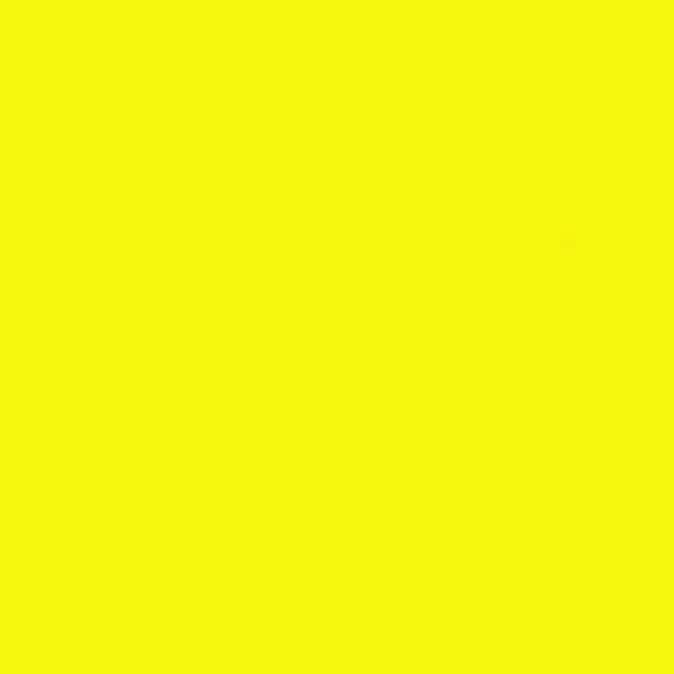}
	
		\caption{Snapshots of a the solution $u_\eps$ of the Allen--Cahn equation (with periodic boundary conditions) after the fast stage in Fig.\ \ref{fig:allencahnfast} on the slower time scale $\sim \eps$. Roughly looks like mean curvature flow (here with periodic boundary conditions).}
		\label{fig:allencahnslow}
	\end{figure}

	The thresholding scheme \cite{MBO}, a highly efficient numerical algorithm, overcomes this drawback in some sense; see Appendix \ref{sec:appendix} for the code of a naive implementation for educational purpose. 
	The idea behind the scheme is intriguingly simple and beautiful: in order to solve the Allen--Cahn equation numerically, apply a standard operator splitting, i.e., given a time-step size $h>0$, alternate between linear diffusion
	\begin{align*}
		\partial_t u^{(1)}_\eps = \Delta u^{(1)}_\eps
	\end{align*}
	for a short time interval $(0,h)$ and pure reaction
	\begin{align*}
		\partial_t u^{(2)}_\eps = \frac1{\eps^2}u^{(2)}_\eps \left(1-(u^{(2)}_\eps)^2\right),
	\end{align*}
	again for a short time interval $(0,h)$.
	The interesting idea now is that Merriman--Bence--Osher \cite{MBO} propose to not solve this ODE but instead \emph{exploit} the smallness $\eps\ll1$ which drives $u_\eps^{(2)}(x,t)$ exponentially fast to the two minima $\pm1$, depending on the sign of $u_\eps^{(2)}(x,0)$ at that particular point $x$ initially. 
	Hence they replace solving this ODE by the simple thresholding step $u\mapsto \sign u$.
	In other words, in each time step, they solve the linear heat diffusion equation and then threshold at the value $0$:
	\begin{align}
		u(\cdot, t+h) = \sign(G_h\ast u(\cdot,t)),
	\end{align}
	where $G_h$ denotes the heat kernel at time $h$. Note that both operations can be implemented efficiently, the convolution by FFT; and the nonlinearity is now (from the implementation point of view) completely trivial.
	Clearly, also this scheme has drawbacks, like pinning for small time-steps, but many of these issues have been overcome in the last decades.

\section{Construction of distributional solutions}
\label{sec:existence}

\subsection{Motivation and definition}

	In this section, we want to define a notion of weak solution to mean curvature flow, originally introduced by Luckhaus--Sturzenhecker \cite{LucStu95} who proved a conditional convergence result for the following implicit time discretization, which has also appeared in \cite{ATW93}:
	
	Given $\Omega_0 \subset \R^d$ and a time-step size $h>0$, for $n=1,2,\ldots$, construct $\Omega_n$ from $\Omega_{n-1}$ by solving
	\begin{align*}
		\mathcal{H}^{d-1} (\p \Omega) + \frac1h \int_{\Omega\Delta \Omega_{n-1}} \dist(x,\p \Omega_{n-1})\dL x.
	\end{align*}

	To motivate the weak solution concept in \cite{LucStu95}, let's first suppose $\Omega(t)$ is a one-parameter family of open sets with smooth boundaries $\Sigma(t)=\p \Omega(t)$, which evolve smoothly by MCF. Then we have
	\begin{align}\label{eq:IBP MC}
		\int_0^T\int_{\Sigma(t)} \left(\nabla \cdot B - \nu\cdot \nabla B\, \nu\right) \dL \H^{d-1} \dL t
	= \int_0^T \int_{\Sigma(t)} H \nu \cdot B  \dL\H^{d-1} \dL t
	\end{align}
	for all test vector fields $B \in C_c^1(\R^d\times(0,T))^d$ and
	\begin{align}\label{eq:velocity smooth}
	\ddt \int_{\Omega(t)}  \zeta \dL x 
	= \int_{\Omega(t)} \p_t\zeta \dL x 
	+ \int_{\Sigma(t)} \zeta V \dL \H^{d-1}
	\end{align}
	for all test functions $\zeta \in C_c^1(\R^d\times[0,T))$. 
	The evolution of the enclosed volume \eqref{eq:velocity smooth} is a straightforward computation. The first equation, \eqref{eq:IBP MC} is a special case ($\p\Sigma=0$) of the well-known integration by parts formula on surfaces. The idea of proof is a combination of Stokes' theorem, which says that the left-hand side applied to the tangential part $B - (B\cdot \nu )\nu$ integrates to zero, and a direct integration by parts for the normal part $(B\cdot \nu ) \nu$.
	
	The two identities \eqref{eq:IBP MC} \& \eqref{eq:velocity smooth} are the basis of the definition  of distributional solutions due to Luckhaus--Sturzenhecker \cite{LucStu95}. 
	Before stating this definition, let us introduce some notation from sets of finite perimeter (or BV functions). We refer the interested reader to \cite{Maggi} for an excellent introduction.

	Here and in the following, we will encode the evolving sets by their characteristic function $\chi(x,t) \in\{0,1\}$. For such a measurable function $\chi \colon [0,1)^d \times (0,T) \to \{0,1\}$, we define 
	\begin{align}
		\int |\nabla \chi| 
		:= \int_{[0,1)^d \times(0,T)} |\nabla \chi| 
		:= \sup_\xi \int_{[0,1)^d\times(0,T)} \chi \left(\nabla \cdot \xi\right) \dL x \dL t ,
	\end{align}
	where the supremum runs over all test vector fields $\xi \in C_c^1([0,1)^d\times(0,T))^d$ with $\sup_{[0,1)^d\times(0,T)} |\xi| \leq 1$. 
	
	If  $\int |\nabla \chi| <\infty$, we say $\chi$ is of \emph{bounded variation} or $\{\chi=1\}$ is a \emph{set of finite perimeter}. 
	Note that $\int |\nabla \chi| <\infty$ is equivalent to the fact that the distributional gradient $\nabla \chi$ is a (vector-valued) Radon measure.
	In that case, one can define the  \emph{measure theoretic outward unit normal} of $\{\chi(\cdot,t)=1\}$ by the Radon-Nikodym derivative $\nu := - \frac{\nabla \chi }{|\nabla \chi|}$.
	Then we have the following generalization of the Gauss-Green formula
	\begin{align}
		\int \chi \left(\nabla \cdot \xi\right) \dL x \dL t = \int \xi \cdot \nu |\nabla \chi| \quad \text{for all } \xi \in C_c^1([0,1)^d\times(0,T))^d.
	\end{align}
	
	The next exercise shows that this formula indeed reduces to the classical Gauss-Green formula if we have sufficient regularity.
	\begin{exercise}
		Suppose $\chi=\chi_\Omega$ for some open set $\Omega \subset [0,1)^d$ with $C^1$ boundary.
		Show that $\nabla \chi= - \nu |\nabla \chi| = -\nu_{\partial \Omega} \mathcal{H}^{d-1}\llcorner \partial \Omega$ and $\int_{[0,1)^d} |\nabla \chi| = \mathcal{H}^{d-1}(\partial \Omega)$.
	\end{exercise}

	One can show even more: If  $\chi\colon [0,1)^d \to \{0,1\}$ with $\int |\nabla \chi| <\infty$, then the measure $|\nabla \chi|$ is concentrated on the \emph{reduced boundary} $\p^* \{\chi=1\} = \{x\colon  \lim_{r\downarrow 0} \frac{\int_{B_r(x)}\nabla \chi}{\int_{B_r(x)}|\nabla \chi|} \text{ exists and is in } \mathbb{S}^{d-1}\}$. Furthermore, De Giorgi's structure theorem (which will not be needed in this course) states that $|\nabla \chi|$ is equal to $\mathcal{H}^{d-1} \llcorner \p^*\{\chi(\cdot,t)=1\}$,  $\p^* \{\chi(\cdot,t)=1\}$ is $\mathcal{H}^{d-1}$-rectifiable, and $\nu$ is indeed the outward unit normal $\mathcal{H}^{d-1}$-a.e., cf.\ \cite[Part 2]{Maggi}.
	
	\begin{definition}[Distributional solution to MCF]
		\label{def:distr sol}
		A measurable function $\chi \colon [0,1)^d\times(0,T) \to \{0,1\}$ is a \emph{distributional solution to MCF} with initial conditions $\chi_0\colon [0,1)^d \to \{0,1\}$  if 
		\begin{align}\label{eq:def distr sol energy}
			\esssup_{t\in (0,T)} \int_{[0,1)^d} |\nabla \chi(\cdot,t)| <\infty
		\end{align}
		and there exists a $|\nabla \chi|$-measurable function $V\colon [0,1)^d\times (0,T)\to \R$ such that
		\begin{align}\label{eq:def distr sol V L2}
			\int V^2 |\nabla \chi| <\infty,
		\end{align}
		which is the normal velocity in the sense that for all test functions $\zeta \in C_c^1([0,1)^d \times [0,T))$ and a.e.\ $T'\in(0,T)$ 
		\begin{align}
		\notag\int_{[0,1)^d} \zeta(\cdot, T')  \chi(\cdot, T') \dL x 
		-&\int_{[0,1)^d} \zeta(\cdot,0) \chi_0 \dL x
		\\\label{eq:def distr sol def V} &= 	\int \chi \,\p_t \zeta \dL x \dL t + \int V \zeta |\nabla \chi|,
		\end{align}
		and such that $V=-H$ in the sense that for all test vector fields $B \in C_c^1(\R^d\times(0,T))^d$ 
		\begin{align}\label{eq:def distr sol V=H}
			\int \left(\nabla \cdot B - \nu\cdot \nabla B\, \nu\right) |\nabla \chi|
			= - \int V \nu \cdot B |\nabla \chi|.
		\end{align}
		
		We say \emph{$\chi$ satisfies the optimal energy-dissipation rate} if in addition for a.e.\ $t\in(0,T)$
		\begin{align}\label{eq:def distr sol EDI}
				\int_{[0,1)^d} |\nabla \chi(\cdot,t)| + \int_{[0,1)^d\times(0,t)} V^2 |\nabla \chi| \leq 	\int_{[0,1)^d} |\nabla \chi_0|. 
		\end{align} 
	\end{definition}

	\begin{remark}
		Note that the initial conditions $\chi(\cdot,0)=\chi_0$ are encoded ``weakly'' in the two statements \eqref{eq:def distr sol def V} and \eqref{eq:def distr sol EDI}. We will see in Section \ref{sec:uniqueness} that this is in a certain sense sufficient.
	\end{remark}

	\begin{exercise}[Consistency]
		Let $\chi(x,t)=\chi_{\Omega(t)}(x)$ be a distributional solution and suppose $\Omega(t)$ is smooth and evolves smoothly. Show that $\p \Omega(t)$ is a classical MCF.
	\end{exercise}

	\begin{exercise}[Approximating normal of finite perimeter sets]
		\label{exercise:approximate normal}
		Suppose $\Omega\subset [0,1)^d$ is a set of finite perimeter and let $\chi := \chi_\Omega$. Show that for each $\eps>0$ there exists $\xi \in C_c^1([0,1)^d)^d$ such that $|\xi|\leq 1$ in $[0,1)^d$ and
		\begin{align}
			 \int |\xi -\nu |^2 |\nabla \chi| < \eps.
		\end{align}
	\end{exercise}

	\subsection{Conditional closure theorem}
	The next statement is the central piece of this section and states a general precompactness result for distributional solutions together with a conditional statement to verify the limiting PDE.
	
	\begin{theorem}[Conditional closure]
		Let $\{\chi_k\}_{k\in \N}$ be  a sequence of distributional solutions to MCF on a common time interval $(0,T)$ in the sense of Definition \ref{def:distr sol} with initial conditions $\chi_{0,k}$ such that 
		\begin{align}
			\chi_{0,k} \to \chi_0 \text{ in }L^1([0,1)^d)\text{ and } \quad \lim_{k\to \infty} \int_{[0,1)^d} |\nabla \chi_{0,k}| =  \int_{[0,1)^d} |\nabla \chi_{0}|
		\end{align}
		and with the uniform bounds
		\begin{align}\label{eq:closure thm uniform energy bound}
			\sup_{k\in \N} \esssup_{t\in (0,T)} \int_{[0,1)^d} |\nabla \chi_k| <\infty
		\end{align}
		and 
		\begin{align}\label{eq:closure thm uniform velocity bound}
			\sup_{k\in \N} \int_{[0,1)^d\times(0,T)} V_k^2 |\nabla \chi_k| <\infty.
		\end{align}
		Then $\{\chi_k\}_{k\in \N}$ is precompact in $L^1([0,1)^d\times(0,T))$, i.e., any sequence contains a subsequence (not relabeled) such that $\chi_k \to \chi$ in $L^1$ for some $\chi \colon[0,1)^d\times (0,T) \to \{0,1\}$.
		
		If in addition the time-integrated energies converge, i.e.,
		\begin{align}\label{eq:closure assumption}
			\lim_{k\to \infty} \int_{[0,1)^d\times(0,T)} |\nabla \chi_k| = \int_{[0,1)^d\times(0,T)} |\nabla \chi|,
		\end{align}
		then the limit $\chi$ is a distributional solution to MCF with initial conditions $\chi_0$.
		
		Finally, if $\chi_k$ satisfies the optimal energy dissipation rate \eqref{eq:def distr sol EDI}, and under the same assumption \eqref{eq:closure assumption}, the limit $\chi$ satisfies the optimal energy dissipation rate as well.
	\end{theorem}

	\begin{proof}
		The proof is divided into four steps. In the first step, we show that under the a priori bounds at hand, we can apply the Riesz-Kolmogorov theorem to obtain compactness. In the second step, we construct the normal velocity and derive optimal bounds for it. In the last two steps, we pass to the limit in the nonlinear terms of the PDE \eqref{eq:def distr sol V=H}. 
		For the sake of simplicity, let us neglect all issues concerning the initial conditions. 
		
		\emph{Step 1: Compactness.} 
			First, we prove the precompactness of $\{\chi_k\}_{k\in \N}$ in $L^1$. 
			By \eqref{eq:closure thm uniform energy bound}, there exists a constant $C<\infty$ independent of $k$ such that
			\begin{align}
			 	\notag\esssup_{t\in(0,T)} &\int_{[0,1)^d} |\chi_k(x+z,t)-\chi_k(x,t)| \dL x 
			 	\\&\label{eq:closure pf compact x} \leq |z|  \esssup_{t\in(0,T)} \int_{[0,1)^d} |\nabla \chi_k|  \leq C|z| \to0\quad \text{as }|z|\to 0.
			\end{align}
			Furthermore, by definition of $V_k$ (i.e., \eqref{eq:def distr sol def V} with $\chi_k$ and $V_k$) and using Cauchy-Schwarz, for any test function $\zeta \in C_c^1([0,1)^d\times(0,T))$, we have
			\begin{align}\label{eq:closure first estimate}
				 \int\zeta \, \p_t \chi_k
				=  \int V_k\zeta  |\nabla \chi_k|
				\leq \left(\int V_k^2 \left|\nabla \chi_k\right|\right)^\frac12 
				\left( \int \zeta^2 \left|\nabla \chi_k\right| \right)^\frac12.
			\end{align}
			Now, for $0<t_1<t_2<T$ and $\zeta \in C_c^1([0,1)^d\times(t_1,t_2))$, we pull out $\sup |\zeta|$ and apply the uniform bounds \eqref{eq:closure thm uniform velocity bound} and \eqref{eq:closure thm uniform energy bound}, to obtain 
			$
			\int \zeta\, \p_t \chi_k \leq C \sup |\zeta| (t_2-t_2)^\frac12
			$
			for some $C<\infty$ independent of $k$ and $\zeta$. Replacing $\zeta$ by $-\zeta$ and taking the supremum over all $\zeta$, we obtain
			\begin{align*}
				\int_{[0,1)^d\times(t_1,t_2)} |\p_t \chi_k| \leq C (t_2-t_2)^\frac12.
			\end{align*}
			Hence we obtain
			\begin{align}\label{eq:closure pf compact t}
				\int_{[0,1)^d}  \left|\chi_k(x,t_2)-\chi_k(x,t_1)\right| \dL x \leq C (t_2-t_2)^\frac12
			\end{align}
			for a.e.\ $0<t_1<t_2<T$.
			
			Therefore, the precompactness follows from \eqref{eq:closure pf compact x} and \eqref{eq:closure pf compact t}, and the Riesz-Kolmogorov compactness theorem.
		
		\emph{Step 2: Construction of velocity.} 
			In this key step, we construct the velocity $V$ of the limit $\chi$ under the assumption \eqref{eq:closure assumption} and show that it satisfies the integrability \eqref{eq:def distr sol V L2} and the defining equation \eqref{eq:def distr sol def V}.
			
			Let us first post-process this assumption: by lower semi-continuity of $\int \zeta |\nabla \chi_k|$ for $\zeta \in C_c([0,1)^d \times(0,T))$ with $\zeta \geq0$, it is easy to see that the assumption \eqref{eq:closure assumption} implies
			\begin{align}\label{eq:conv densities}
				|\nabla \chi_k| \stackrel{\ast}{\rightharpoonup} |\nabla \chi| \quad \text{as Radon measures}.
			\end{align}
			
			We now claim that the measure $\p_t \chi$ satisfies
			\begin{align}\label{eq:closure dt ll dx}
				\p_t \chi \ll |\nabla \chi|.
			\end{align}
			
			We first note that for any $\zeta\in C_c([0,1)^d\times(0,T))$, using \eqref{eq:conv densities} and \eqref{eq:closure thm uniform velocity bound} we may pass to the limit $k\uparrow \infty$ in \eqref{eq:closure first estimate}, and then use $\zeta \mapsto -\zeta$ to get
			\begin{align}\label{eq:closure second estimate}
				\left| \int \zeta \p_t \chi  \right|
				\leq  C
				\left( \int \zeta^2 \left|\nabla \chi\right| \right)^\frac12.
			\end{align}		
			To prove \eqref{eq:closure dt ll dx}, let $\eps>0$ and let $U\subset [0,1)^d \times(0,T)$ be open s.t.\ $\int_{U} |\nabla \chi|<\eps$. 
			Let $\zeta \in C_c^1(U)$ with $|\zeta|\leq 1$ in \eqref{eq:closure second estimate}
			\begin{align*}
				\left| \int \zeta \p_t \chi \right| \leq C \left( \int_U |\nabla \chi| \right)^\frac12 \leq C \eps^\frac12.
			\end{align*}
			Hence $\int_U |\p_t \chi| \leq C \eps^\frac12$ and we obtain \eqref{eq:closure dt ll dx}.
			Therefore, by Radon-Nikodym, there exists a density $V = \frac{\p_t \chi}{|\nabla \chi|} \in L^1([0,1)^d \times(0,T),|\nabla \chi|)$, which by definition satisfies
			\begin{align}
				\int \chi \, \p_t \zeta \dL x\dL t = - \int V\zeta |\nabla \chi| 
			\end{align}
			for all $\zeta \in C_c^1([0,1)^d\times(0,T))$.
			
			To prove the optimal integrability for $V$, we note that \eqref{eq:closure second estimate} can now be written as
			\begin{align*}
				\left| \int \zeta V |\nabla  \chi|  \right| \leq C
				\left( \int \zeta^2 \left|\nabla \chi\right| \right)^\frac12,
			\end{align*}
			so that $V\in L^2$ follows directly from the Riesz representation theorem in $L^2$ (or equivalently take $\zeta \to V$).
			In fact, we could have taken the $\liminf_k$ of the first right-hand side term instead of the $\sup_l$ proof to show the standard lower semi-continuity
			\begin{align*}
				\int V^2 |\nabla \chi| \leq \liminf_{k\to \infty} V_k^2 |\nabla\chi_k|,
			\end{align*}
			which allows us to pass to the limit in \eqref{eq:def distr sol EDI} and show that if each $\chi_k$ satisfies the optimal energy-dissipation rate, then so does $\chi$.
		
		\emph{Step 3: Velocity-term in PDE}
			In order to verify that the limit is a distributional solution, we wish to pass to the limit on both sides of \eqref{eq:def distr sol V=H} separately. 
			Let us first consider the more interesting term involving the velocity. We aim to show that for any fixed $B\in C_c^1([0,1)^d\times(0,T))$
			\begin{align}\label{eq:closure velocity term}
				\lim_{k\to \infty} \int V_k \nu_k\cdot B |\nabla \chi_k| =\int V \nu \cdot B |\nabla \chi|.
			\end{align}
			The structure of this term is weak times weak convergence. However, Lemma \ref{lemma:excess convergence} below shows that the normal $\nu_k$ in fact converges ``strongly'' in some sense.
			
			To exploit this, let $\xi$ be given as in Exercise \ref{exercise:approximate normal}, so that
			\begin{align}
				\mathcal{E} =  \int |\xi -\nu|^2 |\nabla \chi| < \eps.
			\end{align}
			Hence by Lemma \ref{lemma:excess convergence}, for $k$ sufficiently small, it also holds
			\begin{align}
				\mathcal{E}_k =  \int |\xi -\nu_k|^2 |\nabla \chi_k| < \eps.
			\end{align}
			To show \eqref{eq:closure velocity term}, let $\xi \in C_c^1([0,1)^d \times (0,T))^d$ be given. 
			We use the triangle inequality to smuggle in $\xi$ at the expense of $\mathcal{E}$ and $\mathcal{E}_k$ and use $\zeta = \xi \cdot B$ as a test function in \eqref{eq:def distr sol def V} to obtain
			\begin{align*}
				&\left| \int V_k \nu_k\cdot B |\nabla \chi_k| -\int V \nu \cdot B |\nabla \chi|\right| 
				\\&\leq 	\left| \int V_k \xi\cdot B |\nabla \chi_k| -\int V \xi \cdot B |\nabla \chi|\right| 
				\\&\quad \quad +	\left| \int V_k (\nu_k-\xi) \cdot B |\nabla \chi_k| \right|  + \left| \int V (\nu-\xi) \cdot B |\nabla \chi|\right| 
				\\&\leq \left| \int -(\chi_k-\chi)\p_t(\xi\cdot B)\right| 
				\\&\quad\quad+ \|B\|_\infty\left(\int V_k^2 |\nabla \chi_k|\right)^\frac12 \mathcal{E}_k^\frac12 
				+\|B\|_\infty  \left(\int V^2 |\nabla \chi|\right)^\frac12 \mathcal{E}^\frac12.
			\end{align*}
			Since $\chi_k \to \chi $ in $L^1$ as $k\to \infty$ and by the uniform estimates on the velocities \eqref{eq:def distr sol V L2} and \eqref{eq:closure thm uniform velocity bound}, we get
			\begin{align}
				\limsup_{k\to \infty}\left| \int V_k \nu_k\cdot B |\nabla \chi_k| -\int V \nu \cdot B |\nabla \chi|\right| 
				\leq C  \mathcal{E}^\frac12 < C \eps^\frac12.
			\end{align}
			Since the left-hand side does not depend on $\eps$, this shows \eqref{eq:closure velocity term}.
		
		\emph{Step 4: Curvature-term in PDE.} 
		Finally, we claim that the curvature operators converge in the sense that for all $B \in C_c^1([0,1)^d\times(0,T))$ 
		\begin{align}\label{eq:reshetnyak}
			\lim_{k\to \infty} \int \left(\nabla \cdot B - \nu_k\cdot \nabla B\, \nu_k\right) |\nabla \chi_k| 
			= \int \left(\nabla \cdot B - \nu\cdot \nabla B\, \nu\right) |\nabla \chi|.
		\end{align}
		Then the proof is complete as we can now pass to the limit on each side of the weak form of the equation.
		
		This continuity \eqref{eq:reshetnyak} is classical and goes by the name of Reshetnyak's continuity theorem \cite{Reshet}. 
		Here, we give a simple alternative proof using the same methods as in the previous step.
		
		First note that $\nabla \cdot B$ is a valid test function in the weak convergence \eqref{eq:conv densities}, so that we only need to argue for the nonlinear term
		\begin{align}\label{eq:reshetnyak 2}
				\lim_{k\to \infty} \int  \nu_k\cdot \nabla B\, \nu_k |\nabla \chi_k| 
			= \int  \nu\cdot \nabla B\, \nu |\nabla \chi|.
		\end{align}
		
		Let $A:=\nabla B$ denote the Jacobian matrix of $B$. To show \eqref{eq:reshetnyak 2}, as in Step 3, let $\eps>0$ be given and choose a vector field $\xi$ such that $\mathcal{E} <\eps$. 
		Then we estimate
		\begin{align*}
			&\left|\int \nu_k\cdot A\, \nu_k |\nabla \chi_k| 
			- \int  \nu\cdot A\, \nu |\nabla \chi| \right|
			\\&\leq \left|\int \xi \cdot A\, \nu_k |\nabla \chi_k| 
			- \int  \xi\cdot A\, \nu |\nabla \chi| \right|
			\\&\quad \quad +\left|\int (\nu_k-\xi)\cdot A\, \nu_k |\nabla \chi_k| \right|
			+\left| \int  (\nu-\xi) \cdot A\, \nu |\nabla \chi| \right|
			\\&\leq  \left|\int (\chi_k-\chi)\nabla \cdot\left( \xi \cdot A\right) \dL x \dL t\right|
			\\&\quad\quad+\|A\|_\infty\left(\int |\nabla \chi_k| \right)^\frac12 \mathcal{E}_k^\frac12 +	\|A\|_\infty\left(\int |\nabla \chi| \right)^\frac12 \mathcal{E}^\frac12.
		\end{align*}
		By the $L^1$ convergence $\chi_k \to \chi$, the uniform bounds \eqref{eq:def distr sol energy} and \eqref{eq:closure thm uniform energy bound}, and Lemma \ref{lemma:excess convergence}, we conclude
		\begin{align*}
			\limsup_{k\to \infty} \left|\int \nu_k\cdot A\, \nu_k |\nabla \chi_k| 
			- \int  \nu\cdot A\, \nu |\nabla \chi| \right| < CT^{\frac12} \mathcal{E}^\frac12 < CT^\frac12 \eps^\frac12,
		\end{align*}
		which yields the desired \eqref{eq:reshetnyak 2} by taking $\eps\downarrow 0$.
	\end{proof}
	
	The following simple lemma shows that under the assumption of energy convergence, the normals converge strongly in the sense that the tilt-excess $\mathcal{E}$ converges.
	\begin{lemma}\label{lemma:excess convergence}
		Suppose $\chi_k, \chi \in\{0,1\}$ are such that $\chi_k \to \chi$ in $L^1([0,1)^d \times (0,T))$ and such that $\lim_{k\to \infty} \int_{[0,1)^d\times(0,T)} |\nabla \chi_k| = \int_{[0,1)^d\times(0,T)} |\nabla \chi|$, let $\xi \in C_c^1([0,1)^d\times(0,T))^d$, and set 
		\begin{align}
			\mathcal{E}_k
			&:=  \int_{[0,1)^d\times(0,T)} |\nu_k - \xi|^2 |\nabla \chi_k|
			\\\mathcal{E} 
			&:=  \int_{[0,1)^d\times(0,T)} |\nu - \xi|^2 |\nabla \chi|.
		\end{align}
		Then $\lim_{k\to \infty} \mathcal{E}_k = \mathcal{E}.$
	\end{lemma}

	\begin{proof}
		Expand the square $|\nu_k - \xi|^2 = 1 +|\xi|^2 - 2 \xi \cdot \nu_k$, integrate against $|\nabla \chi_k|$, and integrate by parts the last term
		\begin{align}
			\mathcal{E}_k = \int (1+|\xi|^2) |\nabla \chi_k| - 2 \int \chi_k(\nabla \cdot \xi) \dL x \dL t.
		\end{align}
		The two integrals converge to the expected limits due to \eqref{eq:conv densities} and the $L^1$-convergence $\chi_k\to\chi$. 
		Applying the above calculation in reverse order for the limit yields the claim.
	\end{proof}

	\begin{remark}
		It is curious that the energy convergence \eqref{eq:closure assumption} can be derived in the level-set framework, cf.\ \cite{EvansSpruckIV}. The proof is based on a quite fascinating estimate on the mean curvature the form
		\begin{align}
			\sup_{t>0}\int_{[0,1)^d} |H| \left|\nabla u\right|\dL x <\infty,
		\end{align}
		which does not correspond to a geometric levelset-by-levelset estimate. 
		This estimate allows to use the div-curl lemma in a ``vanishing viscosity''-type approximation of the levelset PDE \cite{EvansSpruckIV} and which shows that a.e.\ levelset of the viscosity solution is a unit-density Brakke flow. 
		
		Using similar techniques, one can show that a.e.\ levelset is in fact also a distributional solution in our sense.
	\end{remark}

%
%

	\section{Uniqueness}
	\label{sec:uniqueness}
	\subsection{Motivation and Soner's approach}
		In the two-phase case, the comparison principle for the viscosity solution allows to prove the uniqueness of this weak solution concept in the absence of fattening. In particular, as long as the classical solution exists, it agrees with the viscosity solution. 
		Our goal here is to show another method of proving such ``weak-strong uniqueness'' results, which does not rely on the comparison principle and can indeed be generalized to multiphase systems \cite{FHLS2D}.
		
		Before introducing the new approach, let's first consider the very successful approach of Soner \cite{Soner}, see also Jerrad--Soner\cite{JerSon} and Lin \cite{Lin} for the case of higher codimension.
		Here, we will illuminate this approach in the flavor of a uniqueness question. This simplifies the argument to considering two MCFs instead of the more complicated (and more interesting!) question in \cite{Soner} of the convergence of diffuse interface models to MCF.  
		The general drawback of this approach is that it only yields an upper bound and therefore only an inclusion (and not uniqueness) principle: the weak solution has to be a subset of the strong solution. 
		Furthermore, it is not clear how to generalize this approach to the multiphase case. 
		Both of these drawbacks can be overcome by the new method, at the cost of a slightly more complicated error functional as we will see later.
				
		\begin{theorem}[Inclusion principle]\label{thm inclusion Soner}
			Let $(\Sigma^*(t))_{t\in[0,T]}$ be a smooth classical MCF and let $\Sigma(t)$ be a one-parameter family of surfaces satisfying Brakke's inequality:
			\begin{align}\label{eq:brakke ineq}
					\ddt \int_{\Sigma(t)} \varphi \dL S 
					\leq \int_{\Sigma(t)} \big( -\varphi H^2 - H\nu \cdot \nabla \varphi + \p_t \varphi \big) \dL S
			\end{align}
			for all test functions $\varphi \geq 0$, $\varphi \in C^1([0,1)^d\times[0,T)$. 
			If $\mathcal{H}^{d-1}(\Sigma(0)\setminus \Sigma^*(0))=0$, then $\mathcal{H}^{d-1}(\Sigma(t)\setminus \Sigma^*(t))=0$ for a.e.\ $t\in (0,T)$.
		\end{theorem}
		
		As a direct consequence there is only one smooth classical solution.
		
		\begin{corollary}[Uniqueness of classical solutions]
			If $(\Sigma^*(t))_{t\in[0,T^*)}$ and $(\Sigma(t))_{t\in[0,T)}$ are two smooth classical MCFs with $\Sigma(0) = \Sigma^*(0)$, then $\Sigma(t) = \Sigma^*(t)$ for all $t\in [0,\min \{T, T^\ast\})$.
		\end{corollary}
	
		\begin{proof} 
			By Exercise \ref{exercise:brakke}, we can apply the theorem to the pairs $(\Sigma,\Sigma^*)$ and $(\Sigma^*,\Sigma)$.			
		\end{proof}
	
		\begin{proof}[Proof of Theorem \ref{thm inclusion Soner}]
			Let $\phi := f(\frac12 \dist^2(x,\Sigma^*(t)))$ be a smoothly truncated version of (half of) the squared distance to the smoothly evolving MCF $\Sigma^*(t)$. Here $f( s) =  s$ for $ 0\leq s \leq \delta/2 $, $f( s)=\delta$ for $ s \geq \delta$ and $f'>0$ for $s<\delta$. 
			Then we define the weighted energy of the other surface $\Sigma(t)$
			\begin{align}
				\mathcal{E}(t) := \int_{\Sigma(t)} \phi \dL S.
			\end{align}
			Clearly $\mathcal{E}(t) \geq0$ with equality if and only if $\H^{d-1}(\Sigma(t)\setminus \Sigma^*(t))=0$. 
			We will show that $\ddt \mathcal{E} \leq C \mathcal{E}$, so that the claim follows from Gronwall's inequality.
			
			Plugging $\varphi=\phi$ into \eqref{eq:brakke ineq} and integrating by parts along $\Sigma(t)$ yields
			\begin{align*}
				\ddt \mathcal{E}(t)
				\leq&  \int_{\Sigma(t)} \big(  \p_t \phi - H\nu \cdot \nabla \phi -\phi H^2 \big) \dL S
				\\=& \int_{\Sigma(t)} \big(  \p_t \phi - \Delta \phi + \nu \cdot \nabla^2 \phi\, \nu-\phi H^2 \big) \dL S.
			\end{align*}
			By 	\eqref{eq:dist transport MCF squared} in Lemma~\ref{lemma:dist transport equation MCF}, we have $\p_t \phi - \Delta \phi    \leq-1  + C\phi $ in $\mathcal{U}_\delta$. 
			Furthermore, one can compute $\nabla^2 \phi \leq I_d$ in $\mathcal{U}_\delta$ in the sense of symmetric bilinear forms, i.e., $\nu \cdot \nabla^2 \phi \,\nu \leq |\nu|^2 =1$ in $\mathcal{U}_\delta$.
			Combining these two inequalities, we obtain the near-field estimate
			 $  \p_t \phi - \Delta \phi + \nu \cdot \nabla^2 \phi\, \nu \leq C\phi$ in $\mathcal{U}_\delta$. 
			 Together with the trivial far-field estimate $ \p_t \phi - \Delta \phi + \nu \cdot \nabla^2 \phi \nu \leq C \leq \frac{C}{\delta^2} \phi$ in $\mathcal{U}_\delta^c$, this yields
			\begin{align*}
				\frac{d}{dt} \mathcal{E}(t)
				\leq \int_{\Sigma(t)} \big(  C\phi -\phi H^2 \big) \dL S \leq C \mathcal{E}(t),
			\end{align*}
			which concludes the proof.
		\end{proof}
	
		\begin{remark}
			Inspecting the above proof, we see that in fact, one can lower the regularity of the second interface $\Sigma(t)$ in Theorem \ref{thm inclusion Soner}: the proof applies line by line if $\Sigma(t)$ is replaced by any Brakke flow, i.e., a one-parameter family of varifolds $\mu_t$ which satisfy \eqref{eq:brakke ineq} distributionally in $t$.  However, the interface $\Sigma(t)$ or parts of it could disappear at any time.
		\end{remark}
		
		\subsection{The weak-strong uniqueness principle}
		In order to capture the distance between the two interfaces, we introduce the functional
		\begin{align*}
			 \int_{\Sigma(t)} \left(1- \xi \cdot \nu \right) \dL \mathcal{H}^{d-1}.
		\end{align*}
		In fact, we will use it to compare a smooth classical MCF $\Sigma^\ast(t) = \partial \Omega^*(t)$ with a distributional solution $\chi$ as defined in the previous section:
		\begin{align}\label{eq:def rel entr}
				\mathcal{E}(t) = \int_{[0,1)^d} \left(1- \xi \cdot \nu \right) |\nabla \chi|,
		\end{align}
		where as before $\nu(\cdot,t)$ denotes the measure theoretic outward normal to $\{\chi(\cdot,t)=1\}$ and where $\xi$ will be a suitable extension of the exterior unit normal of $\Sigma^\ast(t)=\p\Omega^*(t)$. In fact, in the present two-phase case, we can simply take
		\begin{align}\label{eq:def xi}
			\xi (x,t):= \zeta(s(x,t))\nabla s(x,t), \quad \text{with } s(x,t) :=\sdist(x,\Sigma^*(t)),
		\end{align}
		and where $\zeta $ is a cutoff such that $\zeta(\tilde s) = 1-\tilde s^2$ for $|\tilde s| \leq \delta/2$ and $\zeta =0$ for $|\tilde s|\geq \delta$.
		We have the lower bound
		\begin{align}
			\notag 1-\xi \cdot \nu 
			&= \frac12 - \xi \cdot \nu  +\frac12|\xi|^2 + \frac12 (1-|\xi|^2)
			\\& \geq\frac12 |\nu - \xi|^2 +\frac12( 1-|\xi|). \label{eq:lower bound entropy density}
		\end{align}
		Since $ 1-|\xi|\geq0 $ everywhere and $1-|\xi| = 1-\zeta(s) = s^2$ in the region where $|s|\leq \delta$, we therefore see that $\mathcal{E}$ controls the tilt-excess and a (truncated) $L^2$-type distance
		\begin{align}\label{eq:tilt exc}
			\frac12\int_{[0,1)^d}  |\nu - \xi|^2 |\nabla \chi| & \leq  \mathcal{E}(t)
			\\ \int_{[0,1)^d} \min\{s^2,\delta^2/4\} |\nabla \chi| &\leq C \mathcal{E}(t).
			\label{eq:s squared}
		 \end{align}
		Recall that the last functional has also appeared in Soner's proof above.
		Finally, we also extend the velocity vector field $V^* \nu^*$ of $\Sigma^*$ as follows
		\begin{align}\label{eq:def B}
			B := -(\Delta s) \xi \quad \text{with $\xi$ and $s$ as in \eqref{eq:def xi}}.
		\end{align}
				
		\begin{remark}
			The vector field $\xi$ is reminiscent of the well-known concept of calibrations in the static calculus of variations, in particular in the theory of minimal surfaces.
			To express this in our language, let $\chi^*\colon D \to \{0,1\}$. Then such a classical calibration is a vector field $\xi \colon D\to \R^d$ such that 
				\begin{align}
				 &\text{$\xi = \nu^*$ on $\supp |\nabla \chi^*|$},
				 \\&\text{$|\xi |\leq 1 $ in $D$,}
				 \\&\text{$\nabla \cdot \xi =0$ in $D$.}
				\end{align}
			If such a calibration exists, then we say $\chi^*$ is calibrated by $\xi$. 
			In that case, it is easy to show that $\chi^*$ minimizes the perimeter (subject to its own boundary conditions):
			\begin{align}
				\int_D |\nabla \chi^*| \leq \int_D |\nabla \chi| \quad \text{for all $\chi$ such that $\chi = \chi^*$ on $\p D$}.
			\end{align}
			(Here, the boundary condition has to be interpreted as a trace, which for sets of finite perimeter can be equivalently written as $\{\chi=1\}^{(1/2)} \cap \p D= \{\chi^*=1\}^{(1/2)} \cap \p D$, where $E^{(s)}$ denotes the points of density $s$, see e.g.\ \cite{Maggi}.) 
			
			Indeed, by the generalized Gauss-Green formula
			\begin{align*}
				\int_D |\nabla \chi| 
				\geq \int_D\nu \cdot \xi |\nabla \chi|
				=\int_D \chi (\nabla \cdot \xi) \dL x - \int_{\p \Omega \cap \{\chi=1\}} \xi \cdot \nu_{\p\Omega} \dL \mathcal{H}^{d-1}
			\end{align*} 
			with equality for $\chi=\chi^*$. 
			Now observe that the first right-hand side integral vanishes and the second one only depends on the boundary conditions. 
			
			The time-dependent vector field $\xi$ in our context may be viewed as a gradient-flow analog of this concept. 
			The conditions here read
			\begin{align}
				 \label{eq:xi nu}&\text{$\xi = \nu^*$ on $\supp |\nabla \chi^*|$},
				 \\\label{eq:xi short}&\text{$|\xi |\leq \max\{1-s^2,0\} $ in $D\times [0,T]$},
				 \\&\label{eq:div xi}\text{$\nabla \cdot \xi =-B\cdot \xi + O(|s|)$ in $D\times[0,T]$},
				 \\&\label{eq:xi transport}\text{$\p_t\xi + (B\cdot \nabla) \xi + (\nabla B)^T \xi = O(|s|) $ in $D\times[0,T]$},
				\\&\label{eq:norm xi transport}\text{$\p_t|\xi|^2 + (B\cdot \nabla) |\xi|^2= O(s^2) $ in $D\times[0,T]$}.
			\end{align} 
			
			The last two additional transport equations are not needed (and trivially satisfied with $B=0$) in the static case. It is easy to check that in our case, all these conditions are satisfied by $\xi $ and $B$ defined by \eqref{eq:def xi} and \eqref{eq:def B}.
			
			Also in the general multiphase case, our concept in \cite{FHLS2D} is a time-dependent analog of the corresponding concept in minimal surface theory, namely paired calibrations introduced in \cite{LawMor}.
		\end{remark}
	
		The main result is the following.
		
		\begin{theorem}[Weak-strong uniqueness]
			Let $(\Sigma^*(t))_{t\in[0,T]} $ with $\Sigma^*(t) = \p \Omega^*(t)$ be a smooth MCF and let $\chi$ be a distributional solution on $(0,T)$ with optimal energy dissipation rate and with initial conditions $\chi_0 = \chi_{\Omega^*(0)}$. Then
			\begin{align}
				\chi (x,t)= \chi_{\Omega^*(t)}(x) \quad \text{a.e.\ in $[0,1)^d$ for a.e.\ $t\in (0,T)$}.
			\end{align}
		\end{theorem}
		
		The theorem follows basically from the following proposition.
		 
		\begin{proposition}\label{prop:rel entr ineq}
			Let $(\Sigma^*(t))_{t\in[0,T]}$ be a smooth MCF with $\Sigma^*(t) = \p \Omega^*(t)$ and define $\xi$ as above in \eqref{eq:def xi}. Then there exists a finite constant $C=C((\Sigma^*(t))_{t\in[0,T]})<\infty$ such that for any distributional solution to MCF with optimal energy dissipation rate  $\chi \colon [0,1)^d \times (0,T)\to \{0,1\}$ and for a.e.\ $t\in (0,T)$ it holds
			\begin{align}
				\mathcal{E}(t) \leq C e^{Ct} \mathcal{E}_0,
			\end{align}
			where $\mathcal{E}(t)$ is given by \eqref{eq:def rel entr} and $\mathcal{E}_0= \int_{[0,1)^d} (1-\xi \cdot \nu_0) |\nabla \chi_0|$.
		\end{proposition}
		
		Just as in Soner's approach, we want to show $\ddt \mathcal{E}(t) \leq C \mathcal{E}(t)$. 
		Loosely speaking, this amounts to showing that all errors are quadratic in the distance $|s|$ to the classical MCF. Therefore it is not surprising that we will at some point in the proof start to complete squares.
		
		To simplify the proof, in addition to the assumptions \eqref{eq:xi nu}--\eqref{eq:norm xi transport}, we will use that $B$ is parallel to $\xi$ and that $|B | \leq C |\xi|$ by construction, cf.\ \eqref{eq:def B}. 
		This assumption is somewhat unnatural and in fact cannot be true in the multiphase, because it does not allow triple junctions to move. However, in our particular case it makes life easier because it allows us to ignore some additional terms.
		
		\begin{proof}
			Rewriting $\mathcal{E}(t)$, testing \eqref{eq:def distr sol def V} with $\zeta = \nabla \cdot \xi$ and  using \eqref{eq:def distr sol EDI}, we obtain for a.e.\ $t\in (0,T)$
			\begin{align*}
				\mathcal{E}(t)
				&= \int_{[0,1)^d} |\nabla \chi(\cdot,t)| - \int_{[0,1)^d} \chi (\cdot,t)\left(\nabla \cdot \xi(\cdot,t)\right) \dL x
				\\&\leq \int_{[0,1)^d} |\nabla \chi_0| -\int_{[0,1)^d} \chi_0 \left(\nabla \cdot \xi(\cdot,0)\right)\dL x 
				-\int_{[0,1)^d\times (0,t)} V^2 |\nabla \chi|
				\\&\quad  - \int_{[0,1)^d\times(0,t)} \chi \p_t \left(\nabla \cdot  \xi\right) \dL x \dL t'
				-\int_{[0,1)^d\times(0,t)}\left( \nabla \cdot \xi \right)V |\nabla \chi|
				\\& = \mathcal{E}_0
				+\int_{[0,1)^d\times (0,t)} \left( -V^2 -V(\nabla \cdot \xi)- \p_t \xi \cdot \nu \right) |\nabla \chi|.				
			\end{align*}
			In the following, let's omit the domain of integration $[0,1)^d\times(0,t)$ for ease of notation. 	 
			
			Now we add zero in form of the PDE \eqref{eq:def distr sol V=H}, tested with the extended velocity field $B$ defined in \eqref{eq:def B}
			\begin{align*}
				\mathcal{E}(t)
				\leq \mathcal{E}_0
				+\int\big(&-V^2 -V(\nabla \cdot \xi) +  V\nu\cdot B 
				\\&+\nabla \cdot B - \nu \cdot \nabla B \, \nu- \p_t \xi \cdot \nu \big) |\nabla \chi|.			
			\end{align*}
			Completing the squares and smuggling in the term $\nu\cdot (B\cdot \nabla) \xi + \xi \cdot( \nu \cdot \nabla )B$ to make the transport equation for $\xi$ appear yields
			\begin{align*}
				\mathcal{E}(t)
				\leq &\mathcal{E}_0
				-\frac12 \int \left(  (V+\nabla \cdot \xi)^2 +| V\nu - B|^2\right) |\nabla \chi|
				\\&+ \int \frac12\big((\nabla \cdot \xi)^2 +|B|^2\big) |\nabla \chi| 
				\\&+ \int \Big(\nabla \cdot B - \nu \cdot \nabla B \, \nu+\nu\cdot (B\cdot \nabla) \xi + \xi \cdot( \nu \cdot \nabla )B\Big) |\nabla \chi| 
				\\& -\int \left( \p_t \xi + (B\cdot \nabla ) \xi + (\nabla B)^T \xi \right)\cdot \nu  |\nabla \chi|.		
			\end{align*}
			Now we use that $B$ is parallel to $\xi$ in the second line and complete another square, and smuggle $\xi$ into the last line 
			\begin{align*}
				\mathcal{E}(t)
				\leq &\mathcal{E}_0
				-\frac12 \int \left(  (V+\nabla \cdot \xi)^2 +| V\nu - B|^2\right) |\nabla \chi|
				\\&+\int \left(\frac12(\nabla \cdot \xi+\xi \cdot B)^2  + \frac12 \left( 1- |\xi|^2\right)|B|^2- (\nabla \cdot \xi) \xi \cdot B\right)|\nabla \chi|
				\\&+\int \left(\nabla \cdot B - \nu \cdot \nabla B \, \nu+\nu\cdot (B\cdot \nabla) \xi + \xi \cdot( \nu \cdot \nabla )B - \xi \cdot \nabla B \,\xi\right) |\nabla \chi| 
				\\&-\frac12\int \left( \p_t |\xi|^2 + (B\cdot \nabla ) |\xi|^2\right)  |\nabla \chi|	
				\\& -\int \left( \p_t \xi + (B\cdot \nabla ) \xi + (\nabla B)^T \xi \right)\cdot (\nu-\xi)  |\nabla \chi|.	
			\end{align*}
			Now we almost recognize the quadratic expression $-(\nu-\xi) \cdot \nabla B (\nu-\xi)$ on the right-hand side:
			\begin{align*}
				\mathcal{E}(t)
				\leq &\mathcal{E}_0
				-\frac12 \int \left(  (V+\nabla \cdot \xi)^2 +| V\nu - B|^2\right) |\nabla \chi|
				\\&+\int \left(\frac12(\nabla \cdot \xi+\xi \cdot B)^2  + \frac12 \left( 1- |\xi|^2\right)|B|^2\right)|\nabla \chi|
				\\&+\int \big(- (\nabla \cdot \xi) \xi \cdot B+\nabla \cdot B +\nu\cdot (B\cdot \nabla) \xi - \nu \cdot( \xi \cdot \nabla )B  \big)|\nabla \chi| 
				\\&-\int (\nu-\xi) \cdot \nabla B (\nu-\xi) |\nabla \chi| 
				\\&-\frac12\int \left( \p_t |\xi|^2 + (B\cdot \nabla ) |\xi|^2\right)  |\nabla \chi|
				\\& -\int \left( \p_t \xi + (B\cdot \nabla ) \xi + (\nabla B)^T \xi \right)\cdot (\nu-\xi)  |\nabla \chi|.	
			\end{align*}
			Now all terms except for the third line are in the desired form. 
			To control that integral, we set $B \wedge \xi := B\otimes \xi - \xi \otimes B$ and exploit the symmetry
			\begin{align*}
				0 
				&= \int \chi \nabla \cdot \big(\nabla \cdot \left( B\wedge \xi \right) \big)\dL x\dL t'
				= \int  \nu \cdot \big(\nabla \cdot \left( B\wedge \xi \right) \big) |\nabla \chi|
				\\&= \int \big(( \nabla \cdot \xi )\nu \cdot B + \nu \cdot (\xi \cdot \nabla) B -( \nabla \cdot B )\nu \cdot \xi - \nu \cdot (B \cdot \nabla) \xi \big) |\nabla \chi|.
			\end{align*}
			Adding this to the right-hand side and using once more that $B$ is parallel to $\xi$ we get
				\begin{align*}
				\mathcal{E}(t)
				\leq &\mathcal{E}_0
				-\frac12 \int \left(  (V+\nabla \cdot \xi)^2 +| V\nu - B|^2\right) |\nabla \chi|
				\\&+\int \left(\frac12(\nabla \cdot \xi+\xi \cdot B)^2  + \frac12 \left( 1- |\xi|^2\right)|B|^2\right)|\nabla \chi|
				\\&+\int \big((\nabla \cdot \xi) (\nu-\xi ) \cdot B + (\nabla \cdot B) (1-\nu \cdot \xi)  \big)|\nabla \chi| 
				\\&-\int (\nu-\xi) \cdot \nabla B (\nu-\xi) |\nabla \chi| 
				\\&-\frac12\int \left( \p_t |\xi|^2 + (B\cdot \nabla ) |\xi|^2\right)  |\nabla \chi|	
				\\& -\int \left( \p_t \xi + (B\cdot \nabla ) \xi + (\nabla B)^T \xi \right)\cdot (\nu-\xi)  |\nabla \chi|.	
			\end{align*}
			Now all terms on the right-hand side are of the desired order: either we can directly estimate against $C(1-\nu \cdot \xi)$, $Cs^2$, or $C |\nu-\xi|^2$, where $C$ depends on $\xi$ and $B$. The errors in the latter two cases can be post-processed by our estimates \eqref{eq:tilt exc} and \eqref{eq:s squared}.
			This is clear for most terms using conditions \eqref{eq:xi nu}--\eqref{eq:norm xi transport} and \eqref{eq:lower bound entropy density}.
			The only term we need to check more carefully is $(\nabla \cdot \xi) (\nu-\xi ) \cdot B $. Here appears an extra cancellation because $B $ is parallel to $\xi$, namely, $|(\nu-\xi ) \cdot B | \leq \frac{|B|}{|\xi|} |\nu \cdot \xi - |\xi|^2| \leq C (1-\nu \cdot \xi )+C( 1 - |\xi|^2) $.
		\end{proof}
		
		To actually improve Soner's inclusion principle to a uniqueness principle, one needs to add a lower-order term to the functional which takes a similar form as the semi-distance $\int s \left(\chi-\chi_{\Omega^*}\right)  \dL x =\int |s| \left|\chi-\chi_{\Omega^*}\right| \dL x = \int_{\Omega \Delta \Omega^*} |s|  \dL x $ in \cite{ATW93, LucStu95}.
			
		For completeness, let us show how to control this semi-distance
		\begin{align}
			\mathcal{F}(t) := \int_{[0,1)^d} \left(\chi - \chi^*\right) \theta\dL x,
		\end{align}
		where as before $\chi$ is our distributional solution and where $\chi^*(x,t)=\chi_{\Omega^*(t)}(x)$ encodes the strong solution and the weight $\theta$ is given by
		\begin{align}
			\theta(x,t) := f(s(x,t)), \quad \text{with }s(x,t)=\sdist(x,\p \Omega^*(t))
		\end{align}
		and $f$ is again a smooth truncation of the identity. More precisely, $f(-\tilde s) = - f(\tilde s)$ and $f( s) =  s$ for $ 0\leq s \leq \delta/2$, $f(\tilde  s)=\delta$ for $ \tilde s \geq \delta$ and $f'\geq0$ in $\R$. 
		
		Then of course
		\begin{align}\label{eq:Ebulk alt}
				\mathcal{F}(t) =  \int_{[0,1)^d} \left|\chi - \chi^*\right| |\theta| \dL x
		\end{align}
		and hence 
		\begin{align}
			\mathcal{F}(t)=0 \text{ if and only if } \chi(\cdot,t)=\chi^*(\cdot,t)\text{ a.e.\ in }[0,1)^d.
		\end{align}
		Note also that 
		\begin{align}\label{eq:theta transport}
			|\p_t \theta + (B\cdot \nabla) \theta| \leq C |\theta|.
		\end{align}
		
		\begin{lemma}
			Under the same conditions as in Proposition \ref{prop:rel entr ineq}, we have
			\begin{align}
				\mathcal{F}(t) \leq C e^{Ct} (\mathcal{F}_0 + \mathcal{E}_0),
			\end{align}
			where similarly to $\mathcal{E}_0$, $\mathcal{F}_0:= \int_{[0,1)^d} \left(\chi_0(x) - \chi^*(x,0)\right) \theta(x,0)\dL x$. 
		\end{lemma}
		\begin{proof}
			Using \eqref{eq:def distr sol def V}, $\theta(\cdot,t')=0$ on $\supp |\nabla \chi^*(\cdot,t')| = \Sigma^*(t')$, \eqref{eq:theta transport}, \eqref{eq:Ebulk alt}, and Gauss-Green, we obtain (again omitting the domain of integration $[0,1)^d\times(0,t)$)
			\begin{align*}
				\mathcal{F}(t) 
				&= \mathcal{F}(0) + \int \theta V |\nabla \chi|  + \int(\chi-\chi^*) \p_t \theta \dL x \dL t' 
				\\& \leq  \mathcal{F}(0)
				 + \int \theta V |\nabla \chi|  
				 \\&\quad - \int (\chi-\chi^*) (B\cdot \nabla) \theta \dL x \dL t' 
				 + C \int |\chi-\chi^*| | \theta| \dL x \dL t' 
				 \\&= \mathcal{F}(0) + C\int_0^t \mathcal{F}(t')\dL t'
				 + \int\theta V |\nabla \chi|  
				 - \int \theta (B\cdot \nu)  |\nabla \chi| 
				 \\&\leq \mathcal{F}(0) + C\int_0^t \mathcal{F}(t')\dL t'
				 + \frac12 \int \theta^2|\nabla \chi| + \frac12 \int |V\nu -B|^2 |\nabla \chi|  ,
			\end{align*}
			where we have used Young's inequality in the last step.
			Crucially, we can absorb the last right-hand side term by the dissipation of $\mathcal{E}(t)$, i.e., the first line on the right-hand side. Also note that $\theta^2 \leq C \min\{s^2,\delta^2\}$ so that we can estimate the second-to-last term by $C\int_0^t\mathcal{E}(t')\dL t'$.
		\end{proof}

	\appendix
	
	\section{A naive implementation of the thresholding scheme}\label{sec:appendix}
	The following code is an even more direct implementation of thresholding than the one described in \S\S \ref{sec:numerical schemes}, which uses FFT. Instead, here the convolution is carried out explicitly, which is less efficient but, as I believe, more transparent. 
	I would recommend the interested reader to experiment with varying initial conditions. 
	
	To run the program from your terminal, save this text to a file, say simplemcf.py, navigate to the folder containing it and enter
	\begin{align*}
	\texttt{python simplemcf.py (grid resolution) (time steps)}
	\end{align*}
	in the command line.
	Here \texttt{grid resolution} should be a power of 2 (e.g. 32),
	\texttt{time steps} should be a positive integer (e.g. 20), for example
	\begin{align*}
	\texttt{python simplemcf.py 32 20}
	\end{align*}
	
	The complete code reads as follows. Note that most of the lines are preparatory steps and the actual algorithm only consists of the convolution and thresholding inside the $t$-loop.
	
	\begin{minted}[obeytabs=true,tabsize=2]{python}
	
import numpy as np
import matplotlib.pyplot as plt
import sys
	
#############################
#   parameters
#############################
	
if len(sys.argv) < 2:
print('''Not enough input arguments.''')
exit(0)
	
n = int(sys.argv[1]) # size of spatial grid
T = int(sys.argv[2]) # number of time steps
width = 4 # width of the kernel (relative to grid size);
	
#############################
#   initial conditions
#############################
	
# random initial conditions
u = np.sign(2*np.random.rand(n,n)-1)

# initializing auxiliary variable
convolution = 0*u

# definition of kernel
# trivial kernel; you can also try out different ones
kernel = np.ones((2*width+1,2*width+1)) 

#############################
#   thresholding scheme
#############################

for t in range(T): #iteration over time steps
	
 # convolution step
 convolution = 0*u
 for x in range(n): #iteration over grid points
  for y in range(n): #iteration over grid points
   for dx in range(-width,width+1): #iteration over nbhd
    for dy in range(-width,width+1): #iteration over nbhd
     if dx**2 + dy**2 < width**2:
      convolution[x,y] += kernel[dx,dy]*u[(x+dx)%n,(y+dy)%n]
	
 # thresholding step
 u = np.sign(convolution)

 # plot current configuration
 plt.imshow(u,interpolation='nearest')
 plt.pause(0.001)
 plt.show()
	\end{minted}
	
	Here are two other examples of initial conditions, which can be used instead of the random one used above:
	
	\begin{minted}[obeytabs=true,tabsize=3]{python}
# square
u = np.ones((n,n))
for x in range(int(round(n/4)), int(round(3*n/4))):
 for y in range(int(round(n/4)), int(round(3*n/4))):
  u[x,y] = -1
	
# ball
u = np.ones((n,n))
for x in range(n):
 for y in range(n):
  if (x-n/2)**2 + (y-n/2)**2 < (n/3)**2:
   u[x,y] = -1
	\end{minted}
	
	\subsection*{Acknowledgements}
	This manuscript is based on (parts of) a one-se\-mes\-ter graduate course held by the author at the University of Bonn in winter 2020/2021 and a short course he held at RWTH Aachen University in spring 2021. 
	Funding by the Deutsche Forschungsgemeinschaft (DFG, German Research Foundation) under Germany's Excellence Strategy -- EXC-2047/1 -- 390685813 is gratefully acknowledged. 
	
	\frenchspacing	
	\bibliographystyle{abbrv}
	\bibliography{references}

\begin{thebibliography}{10}

\bibitem{ATW93}
F.~Almgren, J.~E. Taylor, and L.~Wang.
\newblock Curvature-driven flows: a variational approach.
\newblock {\em SIAM J. Control Optim.}, 31(2):387--438, 1993.

\bibitem{AndrewsBryan}
B.~Andrews and P.~Bryan.
\newblock Curvature bound for curve shortening flow via distance comparison and
  a direct proof of {G}rayson's theorem.
\newblock {\em J. Reine Angew. Math.}, 653:179--187, 2011.

\bibitem{Brakke}
K.~A. Brakke.
\newblock {\em The motion of a surface by its mean curvature}, volume~20 of
  {\em Mathematical Notes}.
\newblock Princeton University Press, Princeton, N.J., 1978.

\bibitem{ChenGigaGoto}
Y.~G. Chen, Y.~Giga, and S.~Goto.
\newblock Uniqueness and existence of viscosity solutions of generalized mean
  curvature flow equations.
\newblock {\em J. Differential Geom.}, 33(3):749--786, 1991.

\bibitem{EvansSpruckI}
L.~C. Evans and J.~Spruck.
\newblock Motion of level sets by mean curvature {I}.
\newblock {\em J. Differential Geom.}, 33(3):635--681, 1991.

\bibitem{EvansSpruckIV}
L.~C. Evans and J.~Spruck.
\newblock Motion of level sets by mean curvature {IV}.
\newblock {\em J. Geom. Anal.}, 5(1):77--114, 1995.

\bibitem{FHLS2D}
J.~Fischer, S.~Hensel, T.~Laux, and T.~Simon.
\newblock The local structure of the energy landscape in multiphase mean
  curvature flow: Weak-strong uniqueness and stability of evolutions.
\newblock {\em Preprint}, 2020.
\newblock arXiv:2003.05478.

\bibitem{GageHamilton}
M.~Gage and R.~S. Hamilton.
\newblock The heat equation shrinking convex plane curves.
\newblock {\em J. Differential Geom.}, 23(1):69--96, 1986.

\bibitem{grayson}
M.~A. Grayson.
\newblock The heat equation shrinks embedded plane curves to round points.
\newblock {\em J. Differential Geom.}, 26(2):285--314, 1987.

\bibitem{HuiskenConvex}
G.~Huisken.
\newblock Flow by mean curvature of convex surfaces into spheres.
\newblock {\em J. Differential Geom.}, 20(1):237--266, 1984.

\bibitem{huiskendistance}
G.~Huisken.
\newblock A distance comparison principle for evolving curves.
\newblock {\em Asian J. Math.}, 2(1):127--133, 1998.

\bibitem{JerSme}
R.~L. Jerrard and D.~Smets.
\newblock On the motion of a curve by its binormal curvature.
\newblock {\em J. Eur. Math. Soc. (JEMS)}, 17(6):1487--1515, 2015.

\bibitem{JerSon}
R.~L. Jerrard and H.~M. Soner.
\newblock Scaling limits and regularity results for a class of
  {G}inzburg-{L}andau systems.
\newblock {\em Ann. Inst. H. Poincar\'{e} Anal. Non Lin\'{e}aire},
  16(4):423--466, 1999.

\bibitem{KimTonegawa}
L.~Kim and Y.~Tonegawa.
\newblock On the mean curvature flow of grain boundaries.
\newblock {\em Ann. Inst. Fourier (Grenoble)}, 67(1):43--142, 2017.

\bibitem{LauxLelmi}
T.~Laux and J.~Lelmi.
\newblock De {G}iorgi's inequality for the thresholding scheme with arbitrary
  mobilities and surface tensions.
\newblock {\em Preprint}, 2021.
\newblock arXiv:2101.11663.

\bibitem{LauxOtto}
T.~Laux and F.~Otto.
\newblock Convergence of the thresholding scheme for multi-phase mean-curvature
  flow.
\newblock {\em Calc. Var. Partial Differential Equations}, 55(5):Art. 129, 74,
  2016.

\bibitem{LauxOttoBrakke}
T.~Laux and F.~Otto.
\newblock Brakke's inequality for the thresholding scheme.
\newblock {\em Calc. Var. Partial Differential Equations}, 59(1):Art. 39, 26,
  2020.

\bibitem{LauxSimon}
T.~Laux and T.~M. Simon.
\newblock Convergence of the {A}llen-{C}ahn equation to multiphase mean
  curvature flow.
\newblock {\em Comm. Pure Appl. Math.}, 71(8):1597--1647, 2018.

\bibitem{LawMor}
G.~Lawlor and F.~Morgan.
\newblock Paired calibrations applied to soap films, immiscible fluids, and
  surfaces or networks minimizing other norms.
\newblock {\em Pacific J. Math.}, 166(1):55--83, 1994.

\bibitem{Lin}
F.~H. Lin.
\newblock Complex {G}inzburg--{L}andau equations and dynamics of vortices,
  filaments, and codimension-{$2$} submanifolds.
\newblock {\em Comm. Pure Appl. Math.}, 51(4):385--441, 1998.

\bibitem{LucStu95}
S.~Luckhaus and T.~Sturzenhecker.
\newblock Implicit time discretization for the mean curvature flow equation.
\newblock {\em Calc. Var. Partial Differential Equations}, 3(2):253--271, 1995.

\bibitem{MacPhersonSrolovitz}
R.~D. MacPherson and D.~J. Srolovitz.
\newblock The von {N}eumann relation generalized to coarsening of
  three-dimensional microstructures.
\newblock {\em Nature}, 446(7139):1053--1055, 2007.

\bibitem{Maggi}
F.~Maggi.
\newblock {\em Sets of finite perimeter and geometric variational problems: an
  introduction to geometric measure theory}, volume 135 of {\em Cambridge
  Studies in Advanced Mathematics}.
\newblock Cambridge University Press, Cambridge, 2012.

\bibitem{MBO}
B.~Merriman, J.~K. Bence, and S.~J. Osher.
\newblock Motion of multiple functions: a level set approach.
\newblock {\em J. Comput. Phys.}, 112(2):334--363, 1994.

\bibitem{Mullins}
W.~W. Mullins.
\newblock Two-dimensional motion of idealized grain boundaries.
\newblock {\em J. Appl. Phys.}, 27:900--904, 1956.

\bibitem{OsherSethian}
S.~Osher and J.~A. Sethian.
\newblock Fronts propagating with curvature-dependent speed: algorithms based
  on {H}amilton--{J}acobi formulations.
\newblock {\em J. Comput. Phys.}, 79(1):12--49, 1988.

\bibitem{read1950dislocation}
W.~T. Read and W.~Shockley.
\newblock Dislocation models of crystal grain boundaries.
\newblock {\em Physical Review}, 78(3):275, 1950.

\bibitem{Reshet}
Y.~G. Reshetnyak.
\newblock The weak convergence of completely additive vector-valued set
  functions.
\newblock {\em Siberian Math. J.}, 9:1386--1394, 1968.

\bibitem{Soner}
H.~M. Soner.
\newblock Front propagation.
\newblock In {\em Boundaries, interfaces, and transitions ({B}anff, {AB},
  1995)}, volume~13 of {\em CRM Proc. Lecture Notes}, pages 185--206. Amer.
  Math. Soc., Providence, RI, 1998.

\bibitem{TonegawaBook}
Y.~Tonegawa.
\newblock {\em Brakke's mean curvature flow: an introduction}.
\newblock SpringerBriefs in Mathematics. Springer, Singapore, 2019.

\end{thebibliography}

\end{document}